\PassOptionsToPackage{sort&compress}{natbib}
\documentclass[11pt]{elsarticle}

\usepackage{subcaption,fullpage}
\usepackage{microtype}
\usepackage{amssymb,amsthm,mathtools}
\usepackage{hyperref}
\usepackage[capitalise,nameinlink,noabbrev]{cleveref}

\journal{Journal of Computational Physics}
\crefname{equation}{}{}
\allowdisplaybreaks

\newcommand{\etal}{\emph{et al.}~}
\newcommand{\ie}{i.e.,~}
\newcommand{\eg}{e.g.,~}

\newcommand{\tensor}[1]{\overset\leftrightarrow{#1}}
\newcommand{\state}[1]{\boldsymbol{#1}}
\newcommand{\mat}[1]{\boldsymbol{#1}}

\newcommand{\T}{\mathrm{T}}
\newcommand{\entropy}{\eta}
\newcommand{\entflux}{F}
\newcommand{\topography}{b}
\newcommand{\totalheight}{H}
\newcommand{\vref}{V}
\newcommand{\entVar}{{\state{w}}}
\newcommand{\refelem}{D}
\newcommand{\physelem}{E}
\newcommand{\cartbasis}{\vec{e}}
\newcommand{\covbasis}{\vec{a}}
\newcommand{\contbasis}{\vec{a}\,}

\newtheorem{lemma}{Lemma}
\newtheorem{theorem}{Theorem}
\theoremstyle{definition}
\newtheorem{definition}{Definition}
\theoremstyle{remark}
\newtheorem{remark}{Remark}

\begin{document}

\begin{frontmatter}

\title{Entropy-stable discontinuous spectral-element methods for the spherical shallow water equations in covariant form}

\author[1,2]{\texorpdfstring{Tristan Montoya\corref{cor1}}{Tristan Montoya}}
\ead{tristan.montoya@usask.ca}
\author[3]{Andrés M.~Rueda-Ramírez}
\author[1,4]{Gregor J.~Gassner}
\cortext[cor1]{Corresponding author, \href{https://orcid.org/0000-0002-4259-1449}{ORCiD: \texttt{0000-0002-4259-1449}}.}
\address[1]{Department of Mathematics and Computer Science, University of Cologne, Germany}
\address[2]{Department of Computer Science, University of Saskatchewan, Canada}
\address[3]{School of Aeronautics (ETSIAE), Universidad Politécnica de Madrid, Spain}
\address[4]{Center for Data and Simulation Science, University of Cologne, Germany}

\begin{abstract}
We introduce discontinuous spectral-element methods of arbitrary order that are well balanced, conservative of mass, and conservative or dissipative of total energy (\ie a mathematical entropy function) for a covariant flux formulation of the rotating shallow water equations with variable bottom topography on curved manifolds such as the sphere. The proposed methods are based on a skew-symmetric splitting of the tensor divergence in covariant form, which we implement and analyze within a general flux-differencing framework using tensor-product summation-by-parts operators. Such schemes are proven to satisfy semi-discrete mass and energy conservation on general unstructured quadrilateral grids in addition to well balancing for arbitrary continuous bottom topographies, with energy dissipation resulting from a suitable choice of numerical interface flux. Furthermore, the proposed covariant formulation permits an analytical representation of the geometry and associated metric terms while satisfying the aforementioned entropy stability, conservation, and well-balancing properties without the need to approximate the metric terms so as to enforce discrete metric identities. Numerical experiments on cubed-sphere grids are presented in order to verify the schemes' structure-preservation properties as well as to assess their accuracy and robustness within the context of several standard test cases characteristic of idealized atmospheric flows. Our theoretical and numerical results support the further development of the proposed methodology towards a full dynamical core for numerical weather prediction and climate modelling, as well as broader applications to other hyperbolic and advection-dominated systems of partial differential equations on curved manifolds.
\end{abstract}


\end{frontmatter}

\section{Introduction}

The development of efficient, robust, and flexible numerical methods for solving the shallow water equations in spherical geometry constitutes an important early step in the development of a new dynamical core for the atmosphere or ocean component of a global weather or climate model, with verification through shallow water test cases (for example, those proposed by Williamson \etal \cite{williamson1992standard}, Galewsky \etal \cite{galewsky_barotropic_instability_04}, and Läuter \etal \cite{laeuter_unsteady_analytical_swe_05}) serving as an important precursor to the solution of more complex geophysical flow problems. Consisting of a scalar equation for mass conservation and a vector equation for momentum conservation, the spherical shallow water equations also represent a prototypical example of a nonlinear hyperbolic system of balance laws on a curved manifold, thereby facilitating the development of methods and algorithms applicable to a broad class of similar partial differential equations arising in various scientific and engineering disciplines. Offering excellent parallel scalability and geometric flexibility as well as the ability to achieve low dissipative and dispersive errors on relatively coarse grids, high-order continuous and discontinuous spectral-element methods represent attractive choices for solving such systems of equations and have therefore been a major focus of recent advances in the development of next-generation numerical methods for global weather and climate models.
\par 
Foundational contributions to the development of continuous Galerkin (CG) methods (see, for example, Taylor \etal \cite{taylor_spectral_element_97}, Giraldo \cite{giraldo_spectral_element_shallow_water_01}, and Thomas and Loft \cite{thomas_loft_semi_implicit_sem_02}) and discontinuous Galerkin (DG) methods (see, for example, Giraldo \etal \cite{giraldo_hesthaven_warburton_nodal_dg_shallow_water_02} and Nair \etal \cite{nair_dg_shallow_water_05}) for the spherical shallow water equations have since been extended to three-dimensional nonhydrostatic atmospheric models based on spectral-element dynamical cores\footnote{In this paper, we use the term \emph{spectral-element method} in a general sense to denote any high-order element-based discretization, whether or not it is derived from a standard Galerkin procedure, and whether or not a collocated nodal basis is used. This family of methods includes standard nodal and modal DG and CG schemes as special cases.} at major modelling centres worldwide (see, for example, the review by Marras \etal \cite{marras_galerkin_review_16} and references therein). A major barrier to the widespread adoption of spectral-element approaches, and high-order methods more broadly, is that their efficiency often comes at the price of robustness for nonlinear problems as a result of their inherently low numerical dissipation, which renders such methods susceptible to aliasing-driven instabilities and nonphysical oscillations introduced by under-resolved high-frequency components of the numerical solution. As discussed in \cite[Section~4]{marras_galerkin_review_16}, continuous as well as discontinuous spectral-element methods typically achieve robustness in practice through the use of \emph{ad hoc} stabilization techniques such as modal filtering or artificial viscosity, which must be tuned carefully in order to avoid introducing excessive dissipative error into the numerical solution. Alternatively, overintegration techniques involving the use of a larger number of quadrature points than required for a linear problem (see, for example, Kirby and Karniadakis \cite{kirby_dealiasing_03} and Mengaldo \etal \cite{mengaldo_dg_fr_dealiasing_15}) can be used to reduce aliasing for nonlinear problems. However, overintegration significantly increases a scheme's computational expense, does not generally provide any theoretical assurance of stability, and has been shown to be insufficient in certain cases when applied to under-resolved discretizations of compressible flows (see, for example, Gassner \etal \cite{gassner_winters_kopriva_splitform_dg_sbp_16} and Winters \etal \cite{winters_moura_mengaldo_overintegration_vs_splitform_18}).
\par 
Tracing their origins to the early work of Kreiss and Scherer \cite{kreiss_scherer_sbp_74} and Tadmor \cite{tadmor_entropy_stable_fv_87} as well as more recent developments by LeFloch \etal \cite{lefloch_entropy_conservative_arbitrary_order_02}, Fisher \cite{fisher_phd_thesis_12}, and Gassner \cite{gassner_dgsem_sbp_13}, modern structure-preserving numerical methods based on summation-by-parts (SBP) operators have evolved into a powerful algebraic framework for the construction of schemes which overcome the robustness issues traditionally afflicting high-order discretizations. These developments are reviewed within the context of discontinuous spectral-element methods applied to computational fluid dynamics by Gassner and Winters \cite{gassner_winters_novel_robust_dg_21}, with the two key components being a discrete derivative operator which mimics integration by parts (\ie satisfies the SBP property) and a modified approximation of the flux divergence based on a consistent skew-symmetric splitting or a \emph{flux-differencing} formulation based on specially chosen finite-volume-like fluxes coupling pairs of quadrature points. Such two-point flux functions can be tailored so as to reduce aliasing as well as preserve certain properties of the continuous problem, including conservation or dissipation of mathematical entropy (a strictly convex function generalizing the notion of thermodynamic entropy, corresponding, for example, to the total energy for a shallow water system).
\par
Although entropy-stable discretizations have recently been developed for two-dimensional geophysical flow models based on the vector-invariant form of the spherical shallow water equations (see, for example, Ricardo \etal \cite{ricardo_conservative_stable_dg_spherical_swe_24}), as well as for three-dimensional atmospheric models based on a Cartesian flux formulation (see, for example, Waruszewski \etal \cite{waruszewski_entropy_stable_dg_euler_atmospheric_22}), to the authors' knowledge, the use of a \emph{covariant} flux formulation has not yet been explored within a split-form or flux-differencing framework. In the context of fluid dynamics in curved geometry, the covariant formulation involves the evolution of the contravariant momentum components as prognostic variables and the formulation of the governing equations as a system of balance laws in which divergences of tensor-valued flux functions are formulated in terms of the covariant derivative. This allows for discretizations to be constructed directly on curved manifolds, rather than within a higher-dimensional ambient space, for example, solving systems on two-dimensional surfaces embedded in three-dimensional space using two, rather than three, momentum equations. When extending such formulations to three-dimensional atmospheric models, the covariant form then allows for separation of the horizontal dynamics from the vertical dynamics, where the latter require special numerical treatment due to the finer mesh spacing and fast wave speeds (see, for example, Baldauf \cite{baldauf_dg_hevi_terrain_following_21}). Unlike the vector-invariant form, the covariant form allows for a locally conservative treatment of the momentum balance, up to the terms introduced by the manifold's curvature, and does not require a specialized discretization of the vorticity. These considerations, as well as the fact that discontinuous spectral-element methods for nonhydrostatic atmospheric models in covariant form are the focus of several dynamical cores currently under development for large-scale weather and climate models \cite{baldauf_dg_shallow_water_covariant_20,baldauf_dg_hevi_terrain_following_21,gaudreault_high_order_shallow_water_cubed_sphere_22,kawai_tomita_dg_dynamical_core_25}, motivate the development of entropy-stable covariant formulations as a key contribution towards improving the robustness of atmospheric simulations.
\par 
In this paper, we address the above objective of extending the SBP framework for split-form and entropy-stable discretizations to nonlinear systems of balance laws in covariant form on curved manifolds. Specifically, we introduce discontinuous spectral-element methods for the covariant shallow water equations with variable bottom topography that achieve arbitrary-order accuracy on general unstructured quadrilateral meshes and satisfy semi-discrete conservation of mass, semi-discrete conservation or dissipation of total energy, and well balancing for a constant surface height and zero velocity. This is achieved through a novel flux-differencing discretization of the covariant derivative that generalizes the skew-symmetric split formulation developed by Gassner \etal \cite{gassner_winters_kopriva_shallow_water_16} and Wintermeyer \etal \cite{wintermeyer2017entropy} to curved manifolds, combined with tensor-product spectral-element operators with the SBP property, a suitable treatment of the nonconservative bottom-topography terms, and an entropy-stable interface flux between adjacent elements.
\par 
We now outline the contents and contributions of the remainder of this paper. In \cref{sec:shallow_water_equations}, we focus on the continuous problem, introducing the necessary notation and definitions for working with vectors and tensors on manifolds, the formulation and entropy analysis of the shallow water equations in covariant form, and the derivation of a novel skew-symmetric split form for the covariant derivative which will serve as the starting point for constructing our discretizations. In \cref{sec:discretization}, we present the major components of the schemes, including the mapping from reference to physical space, the tensor-product SBP operators, and the flux-differencing discontinuous spectral-element formulation. In \cref{sec:entropy_stable_covariant}, we introduce entropy-conservative and entropy-stable two-point flux functions based on our proposed skew-symmetric splitting and show that the resulting schemes are consistent and satisfy the aforementioned conservation, entropy stability, and well-balancing properties. Numerical experiments for an unsteady analytical solution \cite[Example~3]{laeuter_unsteady_analytical_swe_05}, the standard isolated mountain test described in \cite[Case~5]{williamson1992standard}, the barotropic instability problem from \cite{galewsky_barotropic_instability_04}, and the Rossby--Haurwitz wave problem described in \cite[Case~6]{williamson1992standard} are presented in \cref{sec:numerical_experiments}, demonstrating the accuracy, structure-preserving properties, and robustness of the proposed schemes for idealized atmospheric test cases on cubed-sphere grids.

\section{Shallow water equations on manifolds}\label{sec:shallow_water_equations}

In this section, we will discuss the formulation and analysis of balance laws on manifolds, with a specific focus on the covariant shallow water equations on general two-dimensional surfaces.

\subsection{Vector and tensor notation}

We begin by introducing some standard notation which will be used in the formulation of balance laws on manifolds. In particular, we will consider the case of a two-dimensional surface $S \subset \mathbb{R}^3$ embedded within a three-dimensional ambient space, which we take to be spanned by the Cartesian unit basis vectors $\cartbasis_x$, $\cartbasis_y$, and $\cartbasis_z$, and equipped with the standard Euclidean dot product and norm. Let us now define $\vec{X}\colon \refelem \to \physelem$ to be a smooth, bijective, and time-invariant mapping from the standard reference domain $\refelem \subset \mathbb{R}^2$ to the physical domain $\physelem \subset S$, which corresponds to an element of a non-overlapping tessellation (\ie a mesh) of $S$, as illustrated in \cref{fig:mapping}. The space of tangent vectors at a point on $\physelem$ is spanned by the \emph{covariant basis vectors}, which are given by
\begin{equation}\label{eq:basis_covariant}
\covbasis_1 \coloneqq \partial_1\vec{X}, \quad \covbasis_2 \coloneqq \partial_2\vec{X},
\end{equation}
where we use $\partial_1$ and $\partial_2$ to denote partial derivatives with respect to the coordinates $\xi^1$ and $\xi^2$ on the reference domain $\refelem$, which correspond to chart coordinates on $S$ induced by the mapping $\vec{X}$. The covariant basis vectors in \cref{eq:basis_covariant} can be used to obtain the \emph{covariant components} of the metric tensor as $G_{ij} \coloneqq \covbasis_i \cdot \covbasis_j$, which constitute the entries of a symmetric positive-definite (SPD) matrix
\begin{equation}\label{eq:metric_tensor_matrix}
\mat{G} \coloneqq \begin{pmatrix}
G_{11} & G_{12} \\ G_{21} & G_{22}
\end{pmatrix}
\end{equation}
characterizing the geometry of the manifold at a given point. The entries of the inverse matrix $\mat{G}^{-1}$ are likewise known as the \emph{contravariant components} of the metric tensor, which we denote by $G^{ij}$.
\par
\begin{figure}[t!]
\centering
\includegraphics[width=0.7\linewidth]{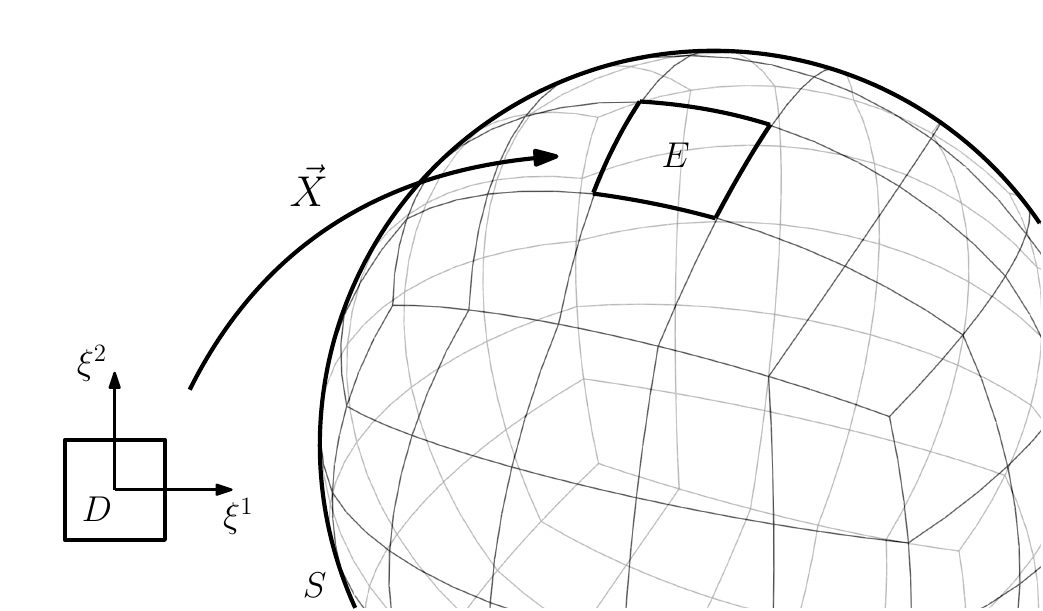}
\caption{Illustration of the surface $S$, reference domain $D$, and physical element $E$}
\label{fig:mapping}
\end{figure}

Employing the summation convention throughout this work for Latin indices repeated across upper and lower positions, any tangent vector to $S$ can be represented as $\vec{v} = v^i\covbasis_i = v^1 \covbasis_1 + v^2 \covbasis_2$ in terms of its contravariant components $v^i$ with respect to the covariant basis vectors $\covbasis_i$, or equivalently as $\vec{v} = v_i \contbasis^i = v_1 \contbasis^1 + v_2 \contbasis^2$ in terms of its covariant components $v_i$ with respect to the contravariant basis vectors $\contbasis^i$. Such components and bases are related by raising and lowering indices as
\begin{equation}\label{eq:raise_lower}
v_i =G_{ij}v^j \iff v^i =G^{ij}v_j, \quad \covbasis_i =G_{ij}\contbasis^j \iff \contbasis^i = G^{ij}\covbasis_j,
\end{equation}
such that the Euclidean dot product and norm can be computed as $\vec{v} \cdot \vec{v} = v_i v^i$ and $\lVert \vec{v} \rVert = \sqrt{v_i v^i}$, respectively. Tensors of second order or higher can be expanded in terms of tensor products of basis vectors, for example, as $\tensor{\tau} = \tau^{ij} \covbasis_i \otimes \covbasis_j$, where the rules in \cref{eq:raise_lower} for raising and lowering indices can be used similarly to convert between covariant, contravariant, and mixed tensor representations.
\par
In order to differentiate a vector field $\vec{v}$ or second-order tensor field $\tensor{\tau}$ along a spatially varying basis vector $\covbasis_k$ given as in \cref{eq:basis_covariant}, we require the \emph{covariant derivative}, which is a tensor field with components given by
\begin{equation}\label{eq:covariant_derivative}
\nabla_k v^i \coloneqq \partial_k v^i + \Gamma_{jk}^i v^j, \quad
\nabla_k\tau^{ij} \coloneqq \partial_k \tau^{ij} + \Gamma_{kl}^i \tau^{lj} + \Gamma_{kl}^j \tau^{il},
\end{equation}
where the \emph{Christoffel symbols of the second kind} are given by
\begin{equation}\label{eq:christoffel_symbols}
\Gamma_{jk}^i \coloneqq
\frac{1}{2}G^{il}\left(\partial_j G_{kl} + \partial_k G_{jl} - \partial_l G_{jk}\right),
\end{equation}
and, following a standard abuse of notation, we emphasize that the operator $\nabla_k$ should be taken as applying to the entire vector field $\vec{v}$ or tensor field $\tensor{\tau}$ rather than to a single component $v^i$ or $\tau^{ij}$. By construction, the covariant derivative satisfies the product rule as well as the \emph{metric compatibility} conditions
\begin{equation}\label{eq:metric_compatibility}
\nabla_kG^{ij} = 0,
\end{equation}
which can be derived in a straightforward manner from \cref{eq:covariant_derivative} and \cref{eq:christoffel_symbols}. The integral of a function $h\colon \physelem \to \mathbb{R}$ can be computed in reference coordinates under the change of variables
\begin{equation}\label{eq:change_of_variables}
\int_{\physelem} h \, \mathrm{d}S = \int_{\refelem} (h \circ \vec{X}) \sqrt{\det \mat{G}}\, \mathrm{d}\xi^1 \mathrm{d}\xi^2,
\end{equation}
where, for convenience of notation, we will define $J \coloneqq \sqrt{\det \mat{G}}$ and suppress composition with $\vec{X}$ when it is unlikely to cause ambiguity. The reader is referred to fluid mechanics texts such as Aris \cite{aris_vectors_tensors_fluid_mechanics} or general relativity texts such as Misner \etal \cite{misner_thorne_wheeler_gravitation} for a more complete introduction to the basic concepts of tensor calculus and differential geometry within the context of continuum physics.

\subsection{Shallow water equations in covariant form}\label{sec:covariant_swe}

Denoting the depth of the fluid layer as $h$ and the velocity vector field as $\vec{v} = v^i \covbasis_i$, the shallow water equations can be formulated on a two-dimensional manifold $S \subset \mathbb{R}^3$ as
\begin{subequations}\label{eq:swe_covariant}
\begin{align}
\partial_t h + \nabla_j (hv^j) &= 0, \label{eq:continuity_covariant}\\
\partial_t (hv^i) + \nabla_j \tau^{ij} &= - fJ G^{ij}\varepsilon_{jk} hv^k - ghG^{ij}\partial_j \topography,\label{eq:momentum_covariant}
\end{align}
\end{subequations}
where $t \geq 0$ is the time coordinate, $g > 0$ is the constant gravitational acceleration, $\topography$ is the spatially varying bottom topography, $f$ is the spatially varying Coriolis parameter, and $\varepsilon_{jk}$ is the Levi-Civita symbol, which takes values of $\varepsilon_{12}=1$, $\varepsilon_{21}=-1$, and $\varepsilon_{11} = \varepsilon_{22} = 0$. The contravariant components of the momentum flux tensor are given by
\begin{equation}\label{eq:momentum_flux}
\tau^{ij} \coloneqq hv^iv^j + \frac{g}{2}h^2G^{ij}.
\end{equation}
Using the definitions in \cref{eq:covariant_derivative} as well as the product rule and the repeated-index identity $\Gamma_{ik}^i = (\partial_k J)/J$,
we can express the divergence of the mass and momentum fluxes as
\begin{equation}\label{eq:flux_divergence}
\nabla_j(hv^j) =\frac{1}{J}\partial_j(Jhv^j), \quad \nabla_j\tau^{ij} =\frac{1}{J}\partial_j(J\tau^{ij}) + \Gamma_{jk}^i \tau^{jk},
\end{equation}
such that the continuity and momentum equations in \cref{eq:continuity_covariant,eq:momentum_covariant}, respectively, can be rewritten using partial derivatives rather than covariant derivatives:
\begin{subequations}\label{eq:pde_reference_space}
\begin{align}
\partial_t h + \frac{1}{J}\partial_j(Jhv^j) &=0,\label{eq:mass_conservative}\\
\partial_t (hv^i) + \frac{1}{J}\partial_j(J\tau^{ij}) &= - fJG^{ij}\varepsilon_{jk}hv^k - ghG^{ij}\partial_j\topography -\Gamma_{jk}^i \tau^{jk}.\label{eq:momentum_conservative}
\end{align}
\end{subequations}
The resulting system of balance laws can then be expressed as
\begin{equation}\label{eq:balance_law_system}
\partial_t\state{u} + \frac{1}{J}\partial_1 (J\state{f}^1) +\frac{1}{J}\partial_2 (J\state{f}^2) = \state{s}_{\mathrm{cor}} + \state{s}_{\mathrm{bot}} + \state{s}_{\mathrm{geo}},
\end{equation}
where the state vector and contravariant flux components are given by
\begin{equation}\label{eq:states_and_fluxes}
\state{u} \coloneqq \begin{pmatrix}
h \\ hv^1 \\ hv^2
\end{pmatrix}, \quad
\state{f}^1 \coloneqq \begin{pmatrix}
hv^1 \\ hv^1v^1 + \frac{g}{2}h^2G^{11} \\ hv^2v^1 + \frac{g}{2}h^2G^{21}
\end{pmatrix}, \quad
\state{f}^2 \coloneqq \begin{pmatrix}
hv^2 \\ hv^1v^2 + \frac{g}{2}h^2G^{12} \\ hv^2v^2 + \frac{g}{2}h^2G^{22}
\end{pmatrix},
\end{equation}
and the Coriolis, bottom topography, and geometric source terms are given, respectively, by
\begin{equation}\label{eq:source_terms_standard}
\state{s}_{\mathrm{cor}} \coloneqq
\begin{pmatrix}
0 \\ -fJG^{1j}\varepsilon_{jk}hv^k \\ -fJG^{2j}\varepsilon_{jk}hv^k
\end{pmatrix}, \quad
\state{s}_{\mathrm{bot}} \coloneqq
\begin{pmatrix}
0 \\ -ghG^{1j}\partial_j \topography\\ -ghG^{2j}\partial_j \topography
\end{pmatrix}, \quad
\state{s}_{\mathrm{geo}} \coloneqq
\begin{pmatrix}
0 \\ -\Gamma_{jk}^1 \tau^{jk} \\ -\Gamma_{jk}^2 \tau^{jk}
\end{pmatrix}.
\end{equation}
This form provides a natural starting point for the construction of DG (see, for example, Läuter \etal \cite{laeuter_spherical_shallow_water_triangular_08}, Bao \etal \cite{bao_flux_form_dg_spherical_swe_14}, Kuang \etal \cite{kuang_evolution_dg_17}, and Baldauf \cite{baldauf_dg_shallow_water_covariant_20}) or finite-volume methods (see, for example, Rossmanith \etal \cite{rossmanith_wave_propagation_sphere_06} and Ullrich \etal \cite{ullrich_finite_volume_shallow_water_10}).

\subsection{Entropy analysis}

It is well known that continuously differentiable solutions to the shallow water equations satisfy an additional conservation law for the total energy per unit area $\entropy \coloneqq hv_iv^i/2 + gh(h+\topography)/2$, which serves as a mathematical \emph{entropy function} for the system (see, for example, Tadmor \cite{tadmor_skew_selfadjoint_84}). Taking the contravariant \emph{entropy flux} components to be $\entflux^j \coloneqq hv_iv^iv^j/2 + gh^2v^j + gh \topography v^j$, such a conservation law is given in covariant form as
\begin{equation}\label{eq:entropy}
\partial_t \entropy + \nabla_j \entflux^j = 0.
\end{equation}
The corresponding \emph{entropy variables} are then obtained by differentiating the entropy function with respect to the state variables as
\begin{equation}\label{eq:entropy_variables}
\entVar \coloneqq \frac{\partial \entropy} {\partial \state{u}} = \begin{pmatrix}
g(h+\topography) - \frac{1}{2}(v_1v^1 + v_2 v^2)\\
v_1 \\
v_2
\end{pmatrix}.
\end{equation}
While the development of a general theory for entropy analysis of hyperbolic systems on manifolds is beyond the scope of this paper (although we note that the scalar case has been treated, for example, by Ben-Artzi and LeFloch \cite{ben_artzi_lefloch_hyperbolic_conservation_laws_manifolds_07}), we will nevertheless verify that, under suitable assumptions, $\entropy$ is strictly convex with respect to the state vector consisting of the fluid layer depth and contravariant momentum components, as required for a mathematical entropy function.

\begin{lemma}\label{lem:strictly_convex}
The mapping from $\state{u}$ to $\entropy$ is strictly convex when $h > 0$ and $\mat{G}$ is SPD.
\end{lemma}

\begin{proof}
Differentiating the entropy variables with respect to the conservative variables, we obtain the Hessian of the entropy function as
\begin{equation}\label{eq:entropy_hessian_covariant}
\frac{\partial \entVar}{\partial \state{u}} = \frac{1}{h}
\begin{pmatrix}
gh + v_1v^1 + v_2 v^2 & - v_1 & - v_2 \\
- v_1 & G_{11} & G_{12} \\
- v_2 & G_{21} & G_{22}
\end{pmatrix}.
\end{equation}
The lower-right two-by-two block of the above matrix is given by $\mat{G}/h$, which is SPD under the present assumptions, and its Schur complement corresponds to the positive scalar $g$. The matrix in \cref{eq:entropy_hessian_covariant} is therefore also SPD, which is equivalent to the strict convexity of $\entropy$ as a function of $\state{u}$.
\end{proof}

\begin{remark}
Due to the role of the total energy as a mathematical entropy function, we will often refer to the quantity $\entropy$ as the ``entropy'' for consistency with the literature on entropy-stable schemes. We will also assume throughout our analysis that the conditions of \cref{lem:strictly_convex} are satisfied regarding the positivity of the fluid layer depth and the positive-definiteness of the metric.
\end{remark}

Whereas continuously differentiable solutions conserve the integral of $\entropy$ within the domain $S$, physically admissible weak solutions satisfy an entropy inequality of the form
\begin{equation}\label{eq:entropy_inequality}
\frac{\mathrm{d}}{\mathrm{d} t}\int_S \entropy \, \mathrm{d}S \leq 0.
\end{equation}
Since establishing such integral balances for a system such as \cref{eq:pde_reference_space} requires the use of the product and chain rules in space, which are not guaranteed to hold discretely, we will begin our construction of an entropy-stable discretization with the introduction of a skew-symmetric splitting which allows for the product and chain rules to be circumvented.

\subsection{Skew-symmetric formulation}\label{sec:skew_symmetric}

Building upon the work of Gassner \etal \cite{gassner_winters_kopriva_shallow_water_16} and Wintermeyer \etal \cite{wintermeyer2017entropy} within the context of the shallow water equations in planar geometry, we prove the following lemma to obtain a skew-symmetric splitting of the covariant-form divergence of the momentum flux.

\begin{lemma}\label{lem:split_form}
For smooth solutions, the divergence of the momentum flux in \cref{eq:momentum_flux} can be expressed as
\begin{equation}\label{eq:split_covariant_derivative}
\begin{aligned}
\nabla_j \tau^{ij} & = \frac{1}{2J} \left(\partial_j(J hv^iv^j) + v^i \partial_j (Jhv^j) + G^{ik}Jhv^j\partial_j v_k\right) + gG^{ij}h\partial_jh \\
& \qquad + \frac{1}{2}\left(\Gamma_{jk}^i hv^jv^k - G^{ik}\Gamma_{jk}^l hv^j v_l \right).
\end{aligned}
\end{equation}
\end{lemma}

\begin{proof}
Distributing the covariant derivative across the convective and pressure terms of \cref{eq:momentum_flux} and noting that the covariant derivative of a scalar reduces to the partial derivative, we can use the product rule and chain rule to obtain
\begin{equation}\label{eq:divergence_separate_pressure_term}
\nabla_j\tau^{ij} = \nabla_j (hv^iv^j) + \frac{1}{2}gh^2\nabla_jG^{ij}+ ghG^{ij}\partial_jh,
\end{equation}
The first term on the right-hand side can then be expressed in two different ways as
\begin{equation}
\begin{aligned}
\nabla_j (hv^iv^j) &= \frac{1}{J}\partial_j(Jhv^iv^j) + \Gamma_{jk}^ihv^jv^k \\
&= \frac{1}{J}\left(v^i\partial_j(Jhv^j)+ G^{ik}Jhv^j\partial_jv_k - G^{ik}\Gamma_{jk}^lhv^jv_l \right),
\end{aligned}
\end{equation}
where the first line results from the use of the second identity in \cref{eq:flux_divergence}, while the second line results from the first identity in \cref{eq:flux_divergence} and several applications of the product rule. Averaging the two expressions and substituting the result into \cref{eq:divergence_separate_pressure_term}, we then note that the second term on the right-hand side of \cref{eq:divergence_separate_pressure_term} is zero due to the metric compatibility condition in \cref{eq:metric_compatibility}. Grouping the differential and algebraic terms therefore results in \cref{eq:split_covariant_derivative}.
\end{proof}

When contracting \cref{eq:split_covariant_derivative} with the covariant velocity components and manipulating the indices using \eqref{eq:raise_lower}, the geometric source terms cancel as
\begin{equation}\label{eq:cancel_christoffel}
v_i(\Gamma_{jk}^ihv^jv^k - G^{ik}\Gamma_{jk}^lhv^jv_l) = v_i \Gamma_{jk}^i hv^jv^k - v^k\Gamma_{jk}^l hv^j v_l = 0,
\end{equation}
resulting in an expression analogous to that obtained in flat (\ie Euclidean) space:
\begin{equation}
v_i \nabla_j \tau^{ij} = \frac{1}{2J} \left(v_i\partial_j(Jhv^iv^j) + v_iv^i \partial_j(Jhv^j) + Jhv^iv^j \partial_j v_i \right) + v^j gh\partial_jh,
\end{equation}
allowing for an integral entropy balance to be derived purely using integration by parts in reference space, a procedure which, as we will see in the following section, is straightforward to mimic at the semi-discrete level. We therefore rewrite the momentum equations in \cref{eq:momentum_conservative} as
\begin{equation}\label{eq:skew_symmetric_momentum}
\begin{aligned}
\partial_t (hv^i) +& \frac{1}{2J} \left(\partial_j(J hv^iv^j) + v^i \partial_j (Jhv^j) + G^{ik}Jhv^j\partial_j v_k\right) + gG^{ij}h\partial_j(h+\topography) \\ & \qquad = - fJG^{ij}\varepsilon_{jk}hv^k - \frac{1}{2}\left(\Gamma_{jk}^i hv^jv^k - G^{ik}\Gamma_{jk}^l hv^j v_l \right),
\end{aligned}
\end{equation}
which constitutes the \emph{skew-symmetric form} of the momentum balance. Together with the continuity equation \cref{eq:continuity_covariant}, this form of the equations will be used in the construction of the entropy-stable discretizations proposed in this work.

\begin{remark}
To the authors' knowledge, skew-symmetric formulations for nonlinear balance laws in covariant form have not been proposed prior to this work.
\end{remark}

\section{Discontinuous spectral-element methods}\label{sec:discretization}

In this section, we will introduce the main components of the spatial discretization framework adopted in this paper, which involves the application of flux-differencing discontinuous spectral-element methods using tensor-product SBP operators to the spatial discretization of balance laws in covariant form.

\subsection{Spherical quadrilateral mapping}\label{sec:mapping}

Let us assume that $S$ is a spherical shell of radius $a$, and consider a global Cartesian coordinate system with its origin at the centre of the sphere. The sphere is tessellated to construct a conforming mesh of curved quadrilateral elements, which can be arbitrary in topology. For each mesh element $\physelem \subset S$, we then let $\vec{x}_1, \vec{x}_2, \vec{x}_3, \vec{x}_4$ denote position vectors from the origin to four distinct points on $S$ defining the corners of a curved element bounded by segments of four great circles. If $\xi^1, \xi^2 \in [-1,1]$ denote coordinates on the reference quadrilateral $\refelem \subset \mathbb{R}^2$, the position vector from the origin to the corresponding point on $\physelem$ is then given by
\begin{equation}\label{eq:mapping}
\vec{X}(\xi^1,\xi^2) \coloneqq a\frac{\vec{x}_e(\xi^1,\xi^2)}{\lVert \vec{x}_e(\xi^1,\xi^2)\rVert},
\end{equation}
where, following Guba \etal \cite[Appendix~A]{guba_variable_resolution_sem_14}, we define the bilinear mapping
\begin{equation}
\begin{aligned}
\vec{x}_e(\xi^1,\xi^2) \coloneqq \frac{1}{4}\Big((1-\xi^1)(1-\xi^2)&\vec{x}_1+ (1+\xi^1)(1-\xi^2)\vec{x}_2 \\ + \, (1+\xi^1)(1+\xi^2)&\vec{x}_3 + (1-\xi^1)(1+\xi^2)\vec{x}_4 \Big)
\end{aligned}
\end{equation}
and note that a similar approach is taken in the case of triangular elements in \cite{laeuter_spherical_shallow_water_triangular_08} and \cite{baldauf_dg_shallow_water_covariant_20}. Differentiating \cref{eq:mapping} as in \cref{eq:basis_covariant}, the covariant basis vectors can then be computed analytically as
\begin{equation}
\covbasis_i(\xi^1,\xi^2) = \frac{a}{\lVert \vec{x}_e(\xi^1,\xi^2)\rVert}\left(\partial_i \vec{x}_e(\xi^1,\xi^2) - \frac{\vec{x}_e(\xi^1,\xi^2) \cdot \partial_i \vec{x}_e(\xi^1,\xi^2)}{\lVert \vec{x}_e(\xi^1,\xi^2)\rVert^2} \vec{x}_e(\xi^1,\xi^2) \right)
\end{equation}
and used to analytically compute the covariant metric components, contravariant metric components, contravariant basis vectors, and Christoffel symbols. We will commonly use $x$, $y$, and $z$ to denote the Cartesian coordinates of position vectors $\vec{x} = x\cartbasis_x + y\cartbasis_y + z\cartbasis_z$ relative to a right-handed coordinate system with its origin at the centre of the sphere. If the positive $x$ axis points towards a longitude of $\lambda =0$ and the positive $z$ axis points towards a latitude of $\theta = \pi/2$, the spherical coordinates corresponding to a position vector $\vec{x}$ can be computed in radians as $\lambda = \operatorname{arctan2}(y,x)$ and $\theta = \arcsin(z/a)$, where we note that $a^2 = \lVert\vec{x}\rVert^2 = x^2 +y^2 + z^2$ and that the two-argument arctangent function is used such that $-\pi \leq \lambda \leq \pi$. In contrast with the common practice of using a global spherical coordinate system to represent the covariant and contravariant basis vectors, we represent such bases using the aforementioned global Cartesian coordinate system. Tangent vectors such as the velocity or momentum can therefore be transformed between local contravariant and global Cartesian component representations as
\begin{equation}\label{eq:global_local_coordinates}
\begin{pmatrix}
v^1 \\
v^2
\end{pmatrix} =
\begin{pmatrix}
\cartbasis_x \cdot \contbasis^1 & \cartbasis_y \cdot \contbasis^1 & \cartbasis_z \cdot \contbasis^1 \\
\cartbasis_x \cdot \contbasis^2 & \cartbasis_y \cdot \contbasis^2 & \cartbasis_z \cdot \contbasis^2
\end{pmatrix}
\begin{pmatrix}
v_x\\
v_y \\
v_z
\end{pmatrix}, \quad
\begin{pmatrix}
v_x\\
v_y \\
v_z
\end{pmatrix} =
\begin{pmatrix}
\cartbasis_x \cdot \covbasis_1 & \cartbasis_x \cdot \covbasis_2 \\
\cartbasis_y \cdot \covbasis_1 & \cartbasis_y \cdot \covbasis_2 \\
\cartbasis_z \cdot \covbasis_1 & \cartbasis_z \cdot \covbasis_2
\end{pmatrix}
\begin{pmatrix}
v^1 \\
v^2
\end{pmatrix},
\end{equation}
where the above transformation matrices are well defined at all points on $S$. This avoids the issue of representing vectors which are nonzero at the poles with respect to a spherical basis.

\subsection{Spectral-element approximation}\label{sec:spectral_element}

To construct a spectral-element spatial discretization on the reference quadrilateral $\refelem$, we will approximate the state vector $\state{u}$ by a tensor-product polynomial expansion of the form
\begin{equation}
\state{u}(\vec{X}(\xi^1,\xi^2),t) \approx \sum_{i=0}^N\sum_{j=0}^N \state{u}_{ij}(t) \ell_i(\xi^1) \ell_j(\xi^2),
\end{equation}
where the \emph{Lagrange polynomials} corresponding to the nodes $\{\xi_i\}_{i=0}^N \subset [-1,1]$ are given by
\begin{equation}
\ell_i(\xi) \coloneqq \prod_{m=0, m\neq i}^N \frac{\xi - \xi_m}{\xi_i - \xi_m}, \quad i \in \{0, \ldots, N\},
\end{equation}
and $\state{u}_{ij}(t) \approx \state{u}(\vec{X}(\xi_i,\xi_j),t)$ are the vector-valued nodal coefficients. In this work, we employ the Legendre--Gauss--Lobatto (LGL) quadrature points, which include both endpoints of the reference interval, and correspond to positive quadrature weights, which we denote by $\{\omega_i\}_{i=0}^N$. We can then use such a quadrature rule to approximate integrals under the change of variables in \cref{eq:change_of_variables} as
\begin{equation}\label{eq:element_integral}
\int_\physelem h \, \mathrm{d}S \approx I_\physelem [h] \coloneqq \sum_{i,j=0}^N \omega_i \omega_j J_{ij} h_{ij}
\end{equation}
where the double subscript $(\cdot)_{ij}$ denotes a quantity evaluated at reference coordinates $(\xi_i, \xi_j)$. By summing \cref{eq:element_integral} over all elements, we can similarly approximate integrals over the global domain $S$ as
\begin{equation}\label{eq:global_integral}
\int_S h \, \mathrm{d}S \approx I_S[h] \coloneqq \sum_{\physelem \subset S} I_\physelem[h].
\end{equation}
Discrete integrals such as those in \cref{eq:element_integral} and \cref{eq:global_integral} will be used extensively in our theoretical analysis and numerical verification of the schemes developed in this paper.

\subsection{Summation-by-parts operators}

Adopting a zero-based indexing convention for discretization matrices, the entries of the collocation derivative matrix $\mat{D}$ and diagonal mass matrix $\mat{M}$ corresponding to the Lagrange polynomials $\{\ell_i\}_{i=0}^N$ and quadrature weights $\{\omega_i\}_{i=0}^N$ are given by
\begin{equation}\label{eq:sbp_matrices}
D_{ij} \coloneqq \left.\frac{\mathrm{d} \ell_j}{\mathrm{d} \xi}\right\rvert_{\xi = \xi_i}, \quad M_{ij} \coloneqq \omega_i \delta_{ij}, \quad i,j\in\{0, \ldots, N\}.
\end{equation}
Defining the matrices $\mat{Q} \coloneqq \mat{M} \mat{D}$ and $\mat{B} \coloneqq \operatorname{diag}(-1, 0, \ldots, 0, 1)$, it was shown by Gassner \cite{gassner_dgsem_sbp_13} that the SBP property
\begin{equation}\label{eq:SBPprop}
\mat{Q} + \mat{Q}^\T = \mat{B}
\end{equation}
is satisfied for collocated spectral-element operators based on LGL quadrature, mimicking the continuous integration-by-parts formula in one spatial dimension, which is given by
\begin{equation}
\int_{-1}^1 u \frac{\mathrm{d} v}{\mathrm{d} \xi}\, \mathrm{d}\xi + \int_{-1}^1 \frac{\mathrm{d} u}{\mathrm{d} \xi}v \, \mathrm{d}\xi = uv\Big\rvert_{-1}^1.
\end{equation}
We also note that the mass matrix $\mat{M}$ is guaranteed to be SPD due to the positivity of the LGL quadrature weights and that the following identity holds due to the exact differentiation of constant functions:
\begin{equation}\label{eq:rowsum}
\sum_{m=0}^N D_{im} = 0, \quad i \in \{0, \ldots, N\}.
\end{equation}
While we use collocated spectral-element operators in this work, the theory applies to any operators satisfying \cref{eq:SBPprop} and \cref{eq:rowsum}, including the finite-difference SBP schemes described, for example, in the review papers by Del Rey Fernández \etal\cite{delrey_sbp_sat_review_14} and Svärd and Nordström \cite{svard_nordstrom_sbpreview_14}.

\subsection{Flux-differencing discretization}\label{sec:flux_differencing}

The discretizations developed in this work will be formulated within the general framework of a nonconservative, variable-coefficient flux-differencing formulation, building on prior work by Renac \cite{renac_entropy_stable_nonconservative_19}, Waruszewski \etal \cite{waruszewski_entropy_stable_dg_euler_atmospheric_22}, and Rueda-Ramírez \etal \cite{ruedaramirez_gauss_collocation_mhd_23}. Defining the skew-symmetric matrix $\mat{S} \coloneqq 2\mat{Q}-\mat{B}$, a flux-differencing tensor-product spatial discretization is given by
\begin{equation}
\begin{aligned} \label{eq:dgsem}
\omega_i\omega_j J_{ij} \frac{\mathrm{d}}{\mathrm{d}t}\state{u}_{ij} &= \omega_j \left(-\sum_{m=0}^N S_{im}(J\state{f}^1)^{\#}_{ij,mj} +
\delta_{i0} (J\state{f}^1)^{*}_{0j,0j^+}
- \delta_{iN}(J\state{f}^1)^{*}_{Nj,Nj^+} \right) \\
&+ \omega_i \left(-\sum_{m=0}^N S_{jm}(J\state{f}^2)^{\#}_{ij,im} +
\delta_{j0} (J\state{f}^2)^{*}_{i0,i0^+}
- \delta_{jN} (J\state{f}^2)^{*}_{iN,iN^+} \right) + \omega_i\omega_j J_{ij}\state{s}_{ij},
\end{aligned}
\end{equation}
where we have suppressed the explicit dependence on time.
Here, \emph{flux differencing} refers to the use of two-point volume flux functions of the form
\begin{equation}
(J\state{f}^1)^{\#}_{L,R} \coloneqq (J\state{f}^1)^{\#}\big((\state{u},\mat{G},\topography)_L, (\state{u},\mat{G},\topography)_R\big), \quad (J\state{f}^2)^{\#}_{L,R} \coloneqq (J\state{f}^2)^{\#}\big((\state{u},\mat{G},\topography)_L, (\state{u},\mat{G},\topography)_R\big),
\end{equation}
as well as interface flux functions of the form
\begin{equation}
(J\state{f}^1)^{*}_{L,R} \coloneqq (J\state{f}^1)^{*}\big((\state{u},\mat{G},\topography)_L, (\state{u},\mat{G},\topography)_R\big), \quad (J\state{f}^2)^{*}_{L,R} \coloneqq (J\state{f}^2)^{*}\big((\state{u},\mat{G},\topography)_L, (\state{u},\mat{G},\topography)_R\big),
\end{equation}
which depend on the solution $\state{u}$ as well as the metric $\mat{G}$ and bottom topography $\topography$ at arbitrary ``left'' and ``right'' states $(\cdot)_L$ and $(\cdot)_R$, and are not necessarily symmetric in their two arguments. For example, $(J\state{f}^1)^{\#}_{ij,mj} = (J\state{f}^1)^{\#}((\state{u},\mat{G},\topography)_{ij}, (\state{u},\mat{G},\topography)_{mj})$ denotes a contravariant volume flux in the $\xi^1$ direction evaluated between the nodes $(\xi_i,\xi_j)$ and $(\xi_m,\xi_j)$, and is not necessarily equal to $(J\state{f}^1)^{\#}_{mj,ij}$.
Any index with a ``$+$'' superscript denotes an external state transformed from an adjacent element, where the local-to-global and global-to-local transformations in \cref{eq:global_local_coordinates} are combined to form the matrix
\begin{equation}\label{eq:transformation_matrix}
\mat{A}_{R\to L} \coloneqq
\begin{pmatrix}
\cartbasis_x \cdot \contbasis^1 & \cartbasis_y \cdot \contbasis^1 & \cartbasis_z \cdot \contbasis^1 \\
\cartbasis_x \cdot \contbasis^2 & \cartbasis_y \cdot \contbasis^2 & \cartbasis_z \cdot \contbasis^2
\end{pmatrix}_L
\begin{pmatrix}
\cartbasis_x \cdot \covbasis_1 & \cartbasis_x \cdot \covbasis_2 \\
\cartbasis_y \cdot \covbasis_1 & \cartbasis_y \cdot \covbasis_2 \\
\cartbasis_z \cdot \covbasis_1 & \cartbasis_z \cdot \covbasis_2
\end{pmatrix}_R,
\end{equation}
which transforms the contravariant representation of the momentum vector with respect to the local coordinate system at $(\cdot)_R$ into the corresponding local coordinate system at $(\cdot)_L$.

\begin{remark}\label{rmk:transform_to_interior}
To simplify our notation and analysis, we will assume that for all interface fluxes, the state vectors on both sides have been transformed so as to include momentum components with respect to the coordinate system associated with the first subscript, corresponding to the interior state, and all geometric quantities used in the flux are evaluated on that side of the interface.
\end{remark}

Using the SBP property in \cref{eq:SBPprop} as well as the row-sum identity in \cref{eq:rowsum}, the weak-form flux-differencing discretization in \cref{eq:dgsem} can be rewritten in strong form as
\begin{equation}\label{eq:dgsem_strong}
\begin{aligned}
\frac{\mathrm{d}}{\mathrm{d}t}\state{u}_{ij} = &-\frac{1}{J_{ij}} \sum_{m=0}^N 2D_{im}(J\state{f}^1)^{\#}_{ij,mj} -
\frac{1}{J_{ij}}\sum_{m=0}^N 2D_{jm}(J\state{f}^2)^{\#}_{ij,im} + \state{s}_{ij}\\
&+\frac{\delta_{i0}}{J_{ij}\omega_i\omega_j}\left[(J\state{f}^1)^{*}_{0j,0j^+} - (J\state{f}^1)_{0j}\right]
- \frac{\delta_{iN}}{J_{ij}\omega_i\omega_j} \left[(J\state{f}^1)^{*}_{Nj,Nj^+} - (J\state{f}^1)_{Nj} \right] \\
 &+ \frac{\delta_{j0}}{J_{ij}\omega_i\omega_j}\left[(J\state{f}^2)^{*}_{i0,i0^+} - (J\state{f}^2)_{i0}\right]
- \frac{\delta_{jN}}{J_{ij}\omega_i\omega_j} \left[(J\state{f}^2)^{*}_{iN,iN^+} - (J\state{f}^2)_{iN}\right],
\end{aligned}
\end{equation}
where we have assumed that two-point flux functions are consistent with \cref{eq:balance_law_system}
such that $(J\state{f})^{\#}_{ij,ij} = (J\state{f})^{*}_{ij,ij} = (J\state{f})_{ij}$. The first line of \cref{eq:dgsem_strong} can then be interpreted as a discretization of the governing equations on a given element, while the second and third lines consist of penalty terms (often referred to as \emph{simultaneous approximation terms} in the SBP literature, following Carpenter \etal \cite{carpenter_time_stable_bcs_94}) which weakly impose the coupling between adjacent elements. Although the weak form in \cref{eq:dgsem} yields a simpler and more efficient implementation and is more easily amenable to entropy analysis, the strong form will be useful in demonstrating consistency and well balancing of the proposed discretizations. Finally, we note that in the specific case that the two-point fluxes are taken to be
\begin{equation}\label{eq:arithmetic_mean_flux}
(J\state{f}^1)^{\#}_{L,R} \coloneqq \frac{1}{2} \left[(J\state{f}^1)_L + (J\state{f}^1)_R \right], \quad (J\state{f}^2)^{\#}_{L,R} \coloneqq \frac{1}{2} \left[(J\state{f}^2)_L + (J\state{f}^2)_R \right],
\end{equation}
and the source term is taken as $\state{s} \coloneqq \state{s}_{\mathrm{cor}} + \state{s}_{\mathrm{bot}} + \state{s}_{\mathrm{geo}}$,
the resulting scheme is equivalent to a standard tensor-product DG method based on collocated LGL quadrature, which is given in weak form by
\begin{equation}
\begin{aligned} \label{eq:dg_standard}
\omega_i\omega_j J_{ij} \frac{\mathrm{d}}{\mathrm{d}t}\state{u}_{ij} &= \omega_j \left(\sum_{m=0}^N Q_{mi}(J\state{f}^1)_{mj} +
\delta_{i0} (J\state{f}^1)^{*}_{0j,0j^+}
- \delta_{iN}(J\state{f}^1)^{*}_{Nj,Nj^+} \right) \\
&+ \omega_i \left(\sum_{m=0}^N Q_{mj}(J\state{f}^2)_{im} +
\delta_{j0} (J\state{f}^2)^{*}_{i0,i0^+}
- \delta_{jN} (J\state{f}^2)^{*}_{iN,iN^+} \right) + \omega_i\omega_j J_{ij}\state{s}_{ij}.
\end{aligned}
\end{equation}
In the following section, we will construct novel formulations of the volume flux, interface flux, and source term so as to ensure consistency, mass conservation, entropy stability, and well balancing of the resulting spatial discretization.

\section{Entropy-stable covariant formulation}\label{sec:entropy_stable_covariant}

Considering the semi-discrete entropy balance on a single element, we can use \cref{eq:element_integral} and apply the chain rule in time to obtain
\begin{equation}\label{eq:entropy_integral}
\frac{\mathrm{d}}{\mathrm{d} t}\int_{\physelem} \entropy \, \mathrm{d}S \approx \frac{\mathrm{d}}{\mathrm{d} t} I_\physelem[\entropy]
=
\sum_{i,j=0}^N \omega_i\omega_j J_{ij}\entVar_{ij}^\T \frac{\mathrm{d}}{\mathrm{d}t} \state{u}_{ij},
\end{equation}
where $\entVar_{ij}$ denotes the vector of entropy variables in \cref{eq:entropy_variables} evaluated in terms of the discrete nodal coefficients $\state{u}_{ij}$. In this section, we will develop an approach for constructing flux-differencing discretizations which satisfy a semi-discrete form of \cref{eq:entropy_inequality} by ensuring that the time derivative of $I_\physelem[\entropy]$ is bounded from above by telescoping interface terms, leading to a semi-discrete bound on the entropy integrated numerically over the global domain $S$. We will also show that the resulting discretizations are conservative of mass and well balanced for arbitrary continuous bottom topographies, preserving a constant surface height and zero velocity.

\subsection{Entropy-conservative and entropy-stable two-point flux functions}

Our next step is to discretize the skew-symmetric formulation proposed in \cref{sec:skew_symmetric} within the flux-differencing framework introduced in \cref{sec:flux_differencing}. The following lemma presents a two-point flux, which, when used with an appropriate formulation of the source term, accomplishes this task.

\begin{lemma}\label{lem:consistent_flux_covariant}
Taking the volume flux and interface flux terms in \cref{eq:dgsem} to be
\begin{equation}\label{eq:flux_ec_covariant}
\begin{aligned}
(J\state{f}^j)^{\#}_{L,R}
\coloneqq
&\begin{pmatrix}
\frac{1}{2}\big[(Jhv^j)_L + (Jhv^j)_R\big] \\
\frac{1}{4}\big[(Jhv^jv^1)_L + (Jhv^jv^1)_R + (Jhv^j)_R(v^1)_L + (G^{1k}Jhv^j)_L(v_k)_R\big]
\\
\frac{1}{4}\big[(Jhv^jv^2)_L + (Jhv^jv^2)_R + (Jhv^j)_R(v^2)_L + (G^{2k}Jhv^j)_L(v_k)_R\big]
\end{pmatrix}
\\
+ &\begin{pmatrix}
0 \\ \frac{g}{2}(G^{1j}Jh)_Lh_R \\ \frac{g}{2}(G^{2j}Jh)_Lh_R
\end{pmatrix}
+ \begin{pmatrix}
0 \\ \frac{g}{2}(G^{1j}Jh)_L(\topography_R - \topography_L) \\ \frac{g}{2}(G^{2j}Jh)_L(\topography_R - \topography_L)
\end{pmatrix}
\end{aligned}
\end{equation}
results in a consistent discretization of \cref{eq:pde_reference_space} when the source term is given by
\begin{equation}\label{eq:source_term_ec}
\state{s} \coloneqq
\begin{pmatrix}
0 \\
-\frac{1}{2}\big(\Gamma_{jk}^1 hv^jv^k - G^{1k}\Gamma_{jk}^l hv^j v_l \big) - f JG^{1j}\varepsilon_{jk} hv^k \\
-\frac{1}{2}\big(\Gamma_{jk}^2 hv^jv^k - G^{2k}\Gamma_{jk}^l hv^j v_l \big) - f JG^{2j}\varepsilon_{jk} hv^k
\end{pmatrix}.
\end{equation}
\end{lemma}

\begin{proof}
Considering the equivalent strong form in \cref{eq:dgsem_strong}, we first note that when all interface states are expressed consistently within a given local coordinate system, the flux differences which constitute the surface terms vanish for globally continuous solutions. Similarly to \cite[Lemma 1]{gassner_winters_kopriva_splitform_dg_sbp_16}, the flux-differencing volume terms can be rewritten as
\begin{align}\label{eq:strong_form_diff_1}
\sum_{m=0}^N 2D_{im}(J\state{f}^1)^{\#}_{ij,mj} &=\sum_{m=0}^N
\begin{pmatrix}
D_{im}(Jhv^1)_{mj} \\
\frac{1}{2}\big[D_{im}(Jhv^1v^1)_{mj} + (v^1)_{ij} D_{im}(Jhv^1)_{mj} + (G^{1k}Jhv^1)_{ij}D_{im}(v_k)_{mj}\big] \\
\frac{1}{2}\big[D_{im}(Jhv^1v^2)_{mj} + (v^2)_{ij} D_{im}(Jhv^1)_{mj} + (G^{2k}Jhv^1)_{ij}D_{im}(v_k)_{mj}\big] \\
\end{pmatrix} \notag\\
&+ \sum_{m=0}^N\begin{pmatrix}
0\\ g(JG^{11}h)_{ij}D_{im}(h + \topography)_{mj} \\
g(JG^{21}h)_{ij}D_{im}(h + \topography)_{mj}
\end{pmatrix}
\end{align}
and
\begin{align}\label{eq:strong_form_diff_2}
\sum_{m=0}^N 2D_{jm}(J\state{f}^2)^{\#}_{ij,im} &=\sum_{m=0}^N
\begin{pmatrix}
D_{jm}(Jhv^2)_{im} \\
\frac{1}{2}\big[D_{jm}(Jhv^2v^1)_{im} + (v^1)_{ij} D_{jm}(Jhv^2)_{im} + (G^{1k}Jhv^2)_{ij}D_{jm}(v_k)_{im}\big] \\
\frac{1}{2}\big[D_{jm}(Jhv^2v^2)_{im} + (v^2)_{ij} D_{jm}(Jhv^2)_{im} + (G^{2k}Jhv^2)_{ij}D_{jm}(v_k)_{im}\big] \\
\end{pmatrix} \notag \\
&+ \sum_{m=0}^N\begin{pmatrix}
0\\ g(JG^{12}h)_{ij}D_{jm}(h + \topography)_{im} \\ g(JG^{22}h)_{ij}D_{jm}(h + \topography)_{im}
\end{pmatrix},
\end{align}
where we have used the fact that terms of the form $D_{im}(\cdot)_{ij}$ and $D_{jm}(\cdot)_{ij}$ do not contribute to sums over $m$ as a consequence of \cref{eq:rowsum}. After dividing by $J_{ij}$, we then see that the right-hand sides of \cref{eq:strong_form_diff_1,eq:strong_form_diff_2} correspond directly to the differential operators appearing in the skew-symmetric formulation, with the partial derivatives in \cref{eq:mass_conservative,eq:skew_symmetric_momentum} replaced by nodal approximations using the SBP derivative matrix in \cref{eq:sbp_matrices}. Noting that the momentum source terms in \cref{eq:source_term_ec} are identical to those appearing on the right-hand side of \cref{eq:skew_symmetric_momentum}, the result therefore follows from \cref{lem:split_form} and the consistency of the spectral-element derivative approximation.
\end{proof}

\begin{remark}
In the special case of the Euclidean metric $G_{ij} = \delta_{ij}$, the two-point flux in \cref{eq:flux_ec_covariant} is the same as that proposed in \cite[Eq.\ (4.3)]{wintermeyer2017entropy} for the shallow water equations in planar geometry. We also note that in addition to making use of split-form approximation of the covariant derivative, the bottom topography term $gG^{ij}\partial_j\topography/2$ is treated as a nonconservative contribution to the two-point flux in
\cref{eq:flux_ec_covariant} rather than as a pointwise source term.
\end{remark}

Next, we present the following lemma establishing that the proposed flux is \emph{entropy conservative} in the sense that it satisfies a particular \emph{shuffle condition}, also known as a \emph{non-symmetric Tadmor condition} (borrowing terminology from Chan \etal \cite{chan_entropy_stable_quasi_1d_24} with reference to \cite{tadmor_entropy_stable_fv_87}), which we will employ in the following subsection to establish entropy stability of the resulting scheme.

\begin{lemma}\label{lem:ec_flux_covariant}
The two-point flux in \cref{eq:flux_ec_covariant} is entropy conservative, satisfying the shuffle condition
\begin{equation}\label{eq:nonsymmetric_tadmor}
\entVar_R^\T
(J\state{f}^j)^{\#}_{R,L} -
\entVar_L^\T
(J\state{f}^j)^{\#}_{L,R}
= (J\varPsi^j)_R - (J\varPsi^j)_L,
\end{equation}
where the contravariant flux potential components are given by $\varPsi^j \coloneqq \entVar^\T\state{f}^j - \entflux^j =
gh^2v^j/2$.
\end{lemma}

\begin{proof}
If we first consider the case of $\topography = 0$ and recall the definition of the entropy variables in \cref{eq:entropy_variables}, the left-hand side of \cref{eq:nonsymmetric_tadmor} can be expanded term-by-term as
\begin{equation}\label{eq:entropy_expanded}
\begin{aligned}
\entVar_R^\T
(J\state{f}^j)^{\#}_{R,L} &-
\entVar_L^\T
(J\state{f}^j)^{\#}_{L,R}
= \underbrace{\frac{g}{2}h_R (Jhv^j)_L}_{(A)}- \frac{g}{2}h_L (Jhv^j)_L
+ \frac{g}{2}h_R(Jhv^j)_R - \underbrace{\frac{g}{2}h_L(Jhv^j)_R}_{(B)}
\\ &- \underbrace{\frac{1}{4}(v_iv^i)_R(Jhv^j)_L}_{(C)} + \underbrace{\frac{1}{4}(v_iv^i)_L (Jhv^j)_L}_{(D)}
- \underbrace{\frac{1}{4}(v_iv^i)_R(Jhv^j)_R}_{(E)} + \underbrace{\frac{1}{4}(v_iv^i)_L(Jhv^j)_R}_{(F)}
\\ &+ \underbrace{\frac{1}{4}(v_i)_R(Jhv^iv^j)_R}_{(E)} - \underbrace{\frac{1}{4}(v_i)_L(Jhv^iv^j)_L}_{(D)}
 + \underbrace{\frac{1}{4}(v_i)_R (Jhv^iv^j)_L}_{(G)} - \underbrace{\frac{1}{4}(v_i)_L (Jhv^iv^j)_R}_{(H)}
\\ &+ \underbrace{\frac{1}{4}(v_i)_R(Jhv^j)_L(v^i)_R}_{(C)} - \underbrace{\frac{1}{4}(v_i)_L(Jhv^j)_R (v^i)_L}_{(F)} + \underbrace{\frac{1}{4}(v_i)_R(G^{ik}Jhv^j)_R(v_k)_L}_{(H)}
\\ &- \underbrace{\frac{1}{4}(v_i)_L(G^{ik}Jhv^j)_L(v_k)_R}_{(G)} + \underbrace{\frac{g}{2}(v_i)_R(G^{ij}Jh)_R h_L}_{(B)} - \underbrace{\frac{g}{2}(v_i)_L(G^{ij}Jh)_Lh_R}_{(A)}.
\end{aligned}
\end{equation}
Noting the cancellation of pairs labelled $(A)$ through $(H)$, we see that the remaining terms constitute the right-hand side of \cref{eq:nonsymmetric_tadmor}. Next, incorporating the influence of the bottom topography $\topography$ within the entropy variables as well as the two-point fluxes, we obtain
\begin{equation}
\begin{aligned}
\entVar_R^\T
(J\state{f}^j)^{\#}_{R,L} -
\entVar_L^\T
(J\state{f}^j)^{\#}_{L,R}
&= (J\varPsi^j)_R - (J\varPsi^j)_L + \underbrace{\frac{g}{2}\topography_R (Jhv^j)_L}_{(A)}
- \underbrace{\frac{g}{2}\topography_L (Jhv^j)_L}_{(B)} + \underbrace{\frac{g}{2}\topography_R(Jhv^j)_R}_{(C)} \\ &- \underbrace{\frac{g}{2}\topography_L(Jhv^j)_R}_{(D)} + \underbrace{\frac{g}{2}(v_i)_R(G^{ij}Jh)_R \topography_L}_{(D)} - \underbrace{\frac{g}{2}(v_i)_L(G^{ij}Jh)_L\topography_R}_{(A)} \\ &- \underbrace{\frac{g}{2}(v_i)_R(G^{ij}Jh)_R \topography_R}_{(C)} + \underbrace{\frac{g}{2}(v_i)_L(G^{ij}Jh)_L\topography_L}_{(B)},
\end{aligned}
\end{equation}
where, similarly to \cref{eq:entropy_expanded}, the pairs labelled $(A)$ through $(D)$ cancel, again resulting in \cref{eq:nonsymmetric_tadmor} when the effect of variable bottom topography is included.
\end{proof}

At element interfaces, we augment the entropy-conservative flux in \cref{eq:flux_ec_covariant} with local Lax--Friedrichs dissipation applied to the jump in the conservative variables in order to obtain
\begin{equation}\label{eq:flux_ec_llf}
(J\state{f}^j)^{*}_{L,R} \coloneqq (J\state{f}^j)^{\#}_{L,R} - \frac{J_L}{2}\max(\lambda_L^j, \lambda_R^j)(\state{u}_R - \state{u}_L),
\end{equation}
where the contravariant wave speed estimate is computed by evaluating the spectral radius of the flux Jacobian $\partial \state{f}^j/\partial\state{u}$ as
\begin{equation}\label{eq:wave_speeds}
\lambda^j \coloneqq \big\lvert v^j \big\rvert + \sqrt{ghG^{jj}}
\end{equation}
for which derivations can be found, for example, in \cite[Appendix~D]{baldauf_dg_shallow_water_covariant_20} and \cite[Appendix~E]{gaudreault_high_order_shallow_water_cubed_sphere_22}. Following Ranocha's analysis in \cite[Section~6.1]{ranocha_comparison_entropy_conservative_fluxes_18}, one can show that the resulting numerical interface flux in \cref{eq:flux_ec_llf} is \emph{entropy stable}, satisfying
\begin{equation}\label{eq:nonsymmetric_tadmor_dissipative}
\entVar_R^\T
(J\state{f}^j)^{*}_{R,L} -
\entVar_L^\T
(J\state{f}^j)^{*}_{L,R} \leq (J\varPsi^j)_R - (J\varPsi^j)_L
\end{equation}
when $h >0$ and the bottom topography $\topography$ is continuous across element interfaces.

\begin{remark}
To show that an entropy-conservative two-point flux augmented with a standard local Lax--Friedrichs dissipation term is entropy stable, Ranocha relies on the Fundamental Theorem of Calculus in \cite[Eq.~(68)]{ranocha_comparison_entropy_conservative_fluxes_18} to express the jump in the conservative variables across an element interface as
\begin{equation}\label{eq:path_integral}
\state{u}_R - \state{u}_L = \left(\int_0^1 \frac{\partial \state{u}}{\partial \entVar} \big((1-\sigma)\entVar_L + \sigma\entVar_R \big)\, \mathrm{d}\sigma \right) (\entVar_R - \entVar_L),
\end{equation}
so as to ensure that the dissipation term does not contribute positively to the entropy balance under the assumption that $\partial \state{u} / \partial \entVar$ is SPD along the entire path from $\entVar_L$ to $\entVar_R$. When the entropy variables $\entVar$ depend on the bottom topography as in \cref{eq:entropy_variables}, however, one cannot invoke \cref{eq:path_integral} across an interface with $b_L \neq b_R$ and thus the arguments in \cite{ranocha_comparison_entropy_conservative_fluxes_18} do not imply entropy stability for discontinuous bottom topographies. While we do not consider discontinuous bottom topographies in this work, one entropy-stable option in that setting is to instead penalize the jump in the entropy variables as
\begin{equation}
(J\state{f}^j)^{*}_{L,R} \coloneqq (J\state{f}^j)^{\#}_{L,R} - \frac{J_L}{2}\max(\lambda_L^j,\lambda_R^j)\mat{H}_{L,R}(\state{w}_R - \state{w}_L),
\end{equation}
where $\mat{H}_{L,R}$ is SPD (or, at least, symmetric positive semi-definite) and typically corresponds to the evaluation of $\partial \state{u} / \partial \entVar$ at an average between the two states. The construction of such a dissipation term was explored in detail by Ranocha \cite{ranocha2017shallow} in the context of the planar shallow water equations.
\end{remark}

\subsection{Mass conservation}\label{sec:conservation}

Noting that the two-point mass flux $(Jhv^j)_{L,R}^\#$ corresponding to the first entry of \cref{eq:flux_ec_covariant} is simply an arithmetic mean, it follows that the resulting flux-differencing discretization of the continuity equation is identical to that used by the standard DG method recovered in \cref{eq:dg_standard}. While we therefore expect Baldauf's proof of conservation in \cite[Appendix~E]{baldauf_dg_shallow_water_covariant_20} to extend to the present formulation, we will proceed more generally with the following theorem, which applies to any symmetric two-point flux for a scalar variable discretized in flux form.

\begin{theorem}\label{thm:mass_balance}
If mass flux components $(Jhv^1)^{\#}_{L,R}$ and $(Jhv^2)^{\#}_{L,R}$ are symmetric with respect to $(\cdot)_L$ and $(\cdot)_R$, and the corresponding entry in the source term vector $\state{s}$ is zero, the semi-discrete mass balance on a single element is then given for the scheme in \cref{eq:dgsem} by
\begin{equation}\label{eq:mass_balance}
\frac{\mathrm{d}}{\mathrm{d}t}I_\physelem[h] = \sum_{j=0}^N \omega_{j} \left[
(Jhv^1)^{*}_{0j,0j^+} -(Jhv^1)^{*}_{Nj,Nj^+}\right]
+ \sum_{i=0}^N \omega_{i} \left[
(Jhv^2)^{*}_{i0,i0^+} - (Jhv^2)^{*}_{iN,iN^+}\right],
\end{equation}
where $(Jhv^1)^{*}_{L,R}$ and $(Jhv^2)^{*}_{L,R}$ are the corresponding interface flux components for the continuity equation.
\end{theorem}

\begin{proof}
Considering the component of \cref{eq:dgsem} corresponding to the continuity equation and summing over all quadrature nodes, we obtain \begin{equation}
\begin{aligned} \label{eq:mass_balance_first_step}
\frac{\mathrm{d}}{\mathrm{d}t} I_\physelem[h] &= \sum_{j=0}^N \omega_j \left(-\sum_{i,m=0}^N S_{im}(Jhv^1)^{\#}_{ij,mj} +
\sum_{i=0}^N\left[\delta_{i0} (Jhv^1)^{*}_{0j,0j^+}
- \delta_{iN}(Jhv^1)^{*}_{Nj,Nj^+} \right]\right) \\
&+ \sum_{i=0}^N \omega_i \left(-\sum_{j,m=0}^N S_{jm}(Jhv^2)^{\#}_{ij,im} +
\sum_{j=0}^N\left[\delta_{j0} (Jhv^2)^{*}_{i0,i0^+}
- \delta_{jN} (Jhv^2)^{*}_{iN,iN^+}\right]\right),
\end{aligned}
\end{equation}
due to the absence of source terms in the mass balance. Noting that the right-hand side of \cref{eq:mass_balance} results from summing over the Kronecker delta, we must show that the contributions from the flux-differencing volume terms vanish. Using the SBP property in \cref{eq:SBPprop} to obtain
\begin{equation}\label{eq:skew_sym}
\mat{S} = \mat{Q} - \mat{Q}^\T,
\end{equation}
we can rewrite the contribution on the first line of \cref{eq:mass_balance_first_step} as
\begin{equation}
\sum_{i,m=0}^N S_{im}(Jhv^1)^{\#}_{ij,mj} =\sum_{i,m=0}^N Q_{im}(Jhv^1)^{\#}_{ij,mj} - \sum_{i,m=0}^N Q_{mi}(Jhv^1)^{\#}_{ij,mj},
\end{equation}
where the first and second terms on the right-hand side cancel when the two-point mass flux is symmetric with respect to its two subscripts. Applying an analogous procedure to the volume term on the second line of \cref{eq:mass_balance_first_step}, we therefore obtain \cref{eq:mass_balance}.
\end{proof}

The right-hand side of \cref{eq:mass_balance} is a discretization of the line integral of the inward normal component of the numerical mass flux over the boundary of the quadrilateral element. To show that the scheme is globally conservative of mass, we proceed similarly to \cite[Appendix~E]{baldauf_dg_shallow_water_covariant_20} under the assumption that every node on the boundary of the element $\physelem$, associated with the state $(\cdot)_L$, coincides in physical space with a node on the boundary of an adjacent element, associated with the state $(\cdot)_R$, for which the corresponding quadrature weight is equal to that associated with $(\cdot)_L$. Summing \cref{eq:mass_balance} over all elements to obtain the global integral in \cref{eq:global_integral} and considering the contribution to the right-hand side from each pair of coincident interface nodes, we see that the discretization is globally conservative when the interface mass fluxes cancel as
\begin{equation}\label{eq:matched_flux}
(Jhv^j)^{*}_{L,R} (n_j)_L = - (Jhv^j)^{*}_{R,L} (n_j)_R
\end{equation}
for all coincident pairs of interface nodes, where $n_j$ denotes the $j^\mathrm{th}$ component of the outward unit normal vector in reference coordinates. Recalling from \cref{rmk:transform_to_interior} that evaluating the left-hand side of \cref{eq:matched_flux} requires transforming the momentum vector at $(\cdot)_R$ into the $(\cdot)_L$ coordinate system, while the right-hand side requires transforming the momentum vector at $(\cdot)_L$ into the $(\cdot)_R$ coordinate system, it follows that the resulting discretization is conservative across element interfaces if the transformation matrices in \cref{eq:transformation_matrix}
satisfy
\begin{equation}\label{eq:inverse_transformations}
\mat{A}_{L\to R} = \mat{A}_{R \to L}^{-1},
\end{equation}
which is the case when the manifold is smooth and the covariant and contravariant basis vectors are computed analytically, as described in \cref{sec:mapping}.

\begin{remark}
In addition to the element-wise conservation property established above, the flux-differencing discretization in \eqref{eq:dgsem} satisfies conservation of mass at the \emph{nodal} level for any symmetric two-point mass flux. This property follows from the results of Fisher \etal \cite{fisher_carpenter_telescopingflux_conservation_13,Fisher2013a} and Carpenter \etal \cite{Carpenter2014}, who showed that flux-differencing formulations using diagonal-norm SBP operators including boundary nodes, such as those used in our LGL-based discontinuous spectral-element method, can be expressed in an equivalent telescoping flux form compatible with the Lax--Wendroff theorem when combined with symmetric two-point fluxes within a one-dimensional or tensor-product discretization. The existence of such an equivalence is important for the use of subcell limiting techniques (see, for example, Rueda-Ramírez \etal \cite{rueda2022subcell}).
\end{remark}

\subsection{Entropy conservation and entropy stability}

We now introduce the following definitions for the numerical entropy flux and entropy production terms, which will be used to facilitate the semi-discrete entropy analysis of the covariant-form discretization through a similar approach to that employed in \cite{ruedaramirez_gauss_collocation_mhd_23} for discretizations of standard Cartesian formulations of nonconservative systems.

\begin{definition}\label{def:entropy_flux_covariant}
The \emph{numerical entropy flux} between two states $(\cdot)_L$ and $(\cdot)_R$ on each side of an element interface is given by
\begin{equation}\label{eq:entropy_flux_covariant}
(J\entflux^j)^{*}_{L,R} \coloneqq
\frac{1}{2}\Big[\entVar_L^\T
(J\state{f}^j)^{*}_{L,R} +
\entVar_{R}^\T
(J\state{f}^j)^{*}_{R,L}\Big]
- \frac{1}{2}\Big[(J\varPsi^j)_L + (J\varPsi^j)_R\Big].
\end{equation}
The \emph{numerical entropy production} terms associated with pairs of volume or interface nodes are given, respectively, by
\begin{subequations}
\begin{align}
r^{j\#}_{L,R} &\coloneqq
\Big[\entVar_R^\T
(J\state{f}^j)^{\#}_{R,L} -
\entVar_L^\T
(J\state{f}^j)^{\#}_{L,R}\Big] - \Big[(J\varPsi^j)_R - (J\varPsi^j)_L\Big], \label{eq:volume_entropy_production}\\
r^{j*}_{L,R} &\coloneqq
\Big[\entVar_R^\T
(J\state{f}^j)^{*}_{R,L} -
\entVar_L^\T
(J\state{f}^j)^{*}_{L,R}\Big] - \Big[(J\varPsi^j)_R - (J\varPsi^j)_L\Big].\label{eq:interface_entropy_production}
\end{align}
\end{subequations}
\end{definition}

\begin{remark}
The entropy conservation and entropy stability conditions in \cref{eq:nonsymmetric_tadmor} and \cref{eq:nonsymmetric_tadmor_dissipative} correspond to $r^{j\#}_{L,R} = 0$ and $r^{j*}_{L,R} \leq 0$, respectively.
\end{remark}

Besides the entropy conservation and entropy stability properties of the two-point flux functions, our proofs of entropy conservation and entropy stability will make use of the fact that the source term yields zero net contribution to the entropy balance, which is established as follows.

\begin{lemma}\label{lem:source_term_neutral}
The source term in \cref{eq:source_term_ec} satisfies $\entVar^\T\state{s} = 0$.
\end{lemma}

\begin{proof}
Contracting the source term in \cref{eq:source_term_ec} with the vector of entropy variables results in
\begin{equation}
\entVar^\T \state{s} = - v_i\left(\frac{1}{2}\Gamma_{jk}^ihv^jv^k - \frac{1}{2}G^{ik}\Gamma_{jk}^lhv^jv_l + f JG^{ij}\varepsilon_{jk} hv^k\right).
\end{equation}
Noting that the contravariant metric tensor components in the last two terms in parentheses raise the covariant index in $v_i$ to obtain the contravariant components $v^k$ and $v^j$, respectively, the contribution from the geometric source term then vanishes as in \cref{eq:cancel_christoffel}, while the contribution from the Coriolis term vanishes due to the skew-symmetry of the Levi-Civita symbol.
\end{proof}

We now present the main result of this section, establishing the semi-discrete entropy balance.

\begin{theorem}\label{thm:entropy_balance_covariant}
Using the numerical entropy flux and production terms introduced in \cref{def:entropy_flux_covariant}, the time rate of change in the mathematical entropy integrated over a given element is given for the spatial discretization in \cref{eq:dgsem} by
\begin{equation}\label{eq:entropy_balance_covariant}
\begin{aligned}
\frac{\mathrm{d}}{\mathrm{d} t} I_\physelem[\entropy]
&=
\sum_{j=0}^N \omega_{j} \left[
(J\entflux^1)^{*}_{0j,0j^+} -(J\entflux^1)^{*}_{Nj,Nj^+}\right]
+ \sum_{i=0}^N \omega_{i} \left[
(J\entflux^2)^{*}_{i0,i0^+} - (J\entflux^2)^{*}_{iN,iN^+}\right]
\\ &+ \frac{1}{2} \sum_{j=0}^N \omega_{j} \left(r^{1*}_{0j^+,0j} + r^{1*}_{Nj,Nj^+}\right)
+ \frac{1}{2}\sum_{i=0}^N \omega_{i} \left(r^{2*}_{i0^+,i0} + r^{2*}_{iN,iN^+}\right) \\ &+ \sum_{i,j,m=0}^N \omega_i\omega_j
\left(
D_{im} r^{1\#}_{ij,mj} + D_{jm} r^{2\#}_{ij,im}
\right).
\end{aligned}
\end{equation}
\end{theorem}

\begin{proof}
Contracting both sides of \cref{eq:dgsem} with the vector of entropy variables $\entVar_{ij}$ and summing over all quadrature nodes, we obtain
\begin{equation}\label{eq:entropy_time_derivative_covariant}
\begin{aligned}
\frac{\mathrm{d}}{\mathrm{d} t} I_\physelem[\entropy] &= \sum_{j=0}^N\omega_j \left(-\sum_{i,m=0}^N \entVar_{ij}^\T S_{im}(J\state{f}^1)^{\#}_{ij,mj} +
\sum_{i=0}^N \entVar_{ij}^\T\left[\delta_{i0} (J\state{f}^1)^{*}_{0j,0j^+}
- \delta_{iN}(J\state{f}^1)^{*}_{Nj,Nj^+}\right] \right) \\
&+ \sum_{i=0}^N\omega_i \left(-\sum_{j,m=0}^N \entVar_{ij}^\T S_{jm}(J\state{f}^2)^{\#}_{ij,im} +
\sum_{j=0}^N \entVar_{ij}^\T \left[\delta_{j0} (J\state{f}^2)^{*}_{i0,i0^+}
- \delta_{jN} (J\state{f}^2)^{*}_{iN,iN^+}\right] \right) \\ &+ \sum_{i,j=0}^N \omega_i\omega_jJ_{ij}\entVar_{ij}^\T\state{s}_{ij}.
\end{aligned}
\end{equation}
Considering the contribution from the flux-differencing volume term on the first line of \cref{eq:entropy_time_derivative_covariant} and using \cref{eq:skew_sym}, we can use the definition of the entropy production term in \cref{eq:volume_entropy_production} to obtain
\begin{equation}
\label{eq:entropy_balance_volume_covariant_1}
\begin{aligned}
\sum_{i,m=0}^N \entVar_{ij}^\T S_{im}(J\state{f}^1)^{\#}_{ij,mj} &= \sum_{i,m=0}^N\entVar_{ij}^\T (Q_{im} - Q_{mi})(J\state{f}^1)^{\#}_{ij,mj} \\
&= \sum_{i,m=0}^N Q_{im}\left[\entVar_{ij}^\T (J\state{f}^1)^{\#}_{ij,mj} - \entVar_{mj}^\T (J\state{f}^1)^{\#}_{mj,ij}\right] \\
&=\sum_{i,m=0}^N Q_{im}\left[(J\varPsi^1)_{mj} - (J\varPsi^1)_{ij} - r^{1\#}_{ij,mj}\right] \\
&= \sum_{i,m=0}^N Q_{im}(J\varPsi^1)_{mj} - \sum_{i=0}^N (J\varPsi^1)_{ij}\sum_{m=0}^NQ_{im} - \sum_{i,m=0}^NQ_{im}r^{1\#}_{ij,mj} \\
&= (J\varPsi^1)_{Nj} - (J\varPsi^1)_{0j} - \sum_{i,m=0}^N\omega_iD_{im}r^{1\#}_{ij,mj},
\end{aligned}
\end{equation}
where the final equality results from the properties in \cref{eq:SBPprop} and \cref{eq:rowsum} as well as the definition of $\mat{Q}$. Following the same procedure for the flux-differencing term on the second line of \cref{eq:entropy_time_derivative_covariant}, we likewise obtain
\begin{equation}\label{eq:entropy_balance_volume_covariant_2}
\sum_{j,m=0}^N \entVar_{ij}^\T S_{jm}(J\state{f}^2)^{\#}_{ij,im} = (J\varPsi^2)_{iN} - (J\varPsi^2)_{i0} - \sum_{j,m=0}^N\omega_jD_{jm}r^{2\#}_{ij,im}.
\end{equation}
We now turn our attention to the interface terms, which contract with the entropy variables as
\begin{equation}\label{eq:entropy_balance_surface_covariant_1}
\begin{multlined}
\sum_{i=0}^N \entVar_{ij}^\T\left[\delta_{i0} (J\state{f}^1)^{*}_{0j,0j^+}
- \delta_{iN}(J\state{f}^1)^{*}_{Nj,Nj^+}\right]
= (J\varPsi^1)_{0j}
- (J\varPsi^1)_{Nj} \\
+(J\entflux^1)^{*}_{0j,0j^+} -(J\entflux^1)^{*}_{Nj,Nj^+} + \frac{1}{2}\left(r^{1*}_{0j^+,0j} + r^{1*}_{Nj,Nj^+}\right)
\end{multlined}
\end{equation}
and
\begin{equation}\label{eq:entropy_balance_surface_covariant_2}
\begin{multlined}
\sum_{j=0}^N \entVar_{ij}^\T \left[\delta_{j0} (J\state{f}^2)^{*}_{i0,i0^+}
- \delta_{jN} (J\state{f}^2)^{*}_{iN,iN^+}\right]
= (J\varPsi^2)_{i0}
- (J\varPsi^2)_{iN} \\
+ (J\entflux^2)^{*}_{i0,i0^+} - (J\entflux^2)^{*}_{iN,iN^+}
+ \frac{1}{2} \left(r^{2*}_{i0^+,i0} + r^{2*}_{iN,iN^+}\right),
\end{multlined}
\end{equation}
where we have used \cref{lem:ec_flux_covariant} as well as \cref{eq:entropy_flux_covariant,eq:interface_entropy_production}. Substituting \cref{eq:entropy_balance_volume_covariant_1}, \cref{eq:entropy_balance_volume_covariant_2}, \cref{eq:entropy_balance_surface_covariant_1}, and \cref{eq:entropy_balance_surface_covariant_2} into \cref{eq:entropy_time_derivative_covariant}, noting the cancellation of the terms $(J\varPsi^1)_{0j}
- (J\varPsi^1)_{Nj}$ and $(J\varPsi^2)_{i0}
- (J\varPsi^2)_{iN}$, and invoking \cref{lem:source_term_neutral} to show that the source term in \cref{eq:source_term_ec} does not contribute to the entropy balance, we therefore obtain \cref{eq:entropy_balance_covariant}.
\end{proof}

\begin{remark}
Unlike existing proofs of entropy stability in curvilinear coordinates (see, for example, \cite[Appendix~B.1.1]{fisher_phd_thesis_12}), the proof of \cref{thm:entropy_balance_covariant} does not rely on any discrete metric identities. This allows us to use an analytical representation of the geometry and tangent basis vectors, which ensures that \cref{eq:inverse_transformations} is satisfied.
\end{remark}

The first line on the right-hand side of \cref{eq:entropy_balance_covariant} represents a discrete integral of the numerical entropy flux directed into the element, the second line contains interface production terms which are zero for an entropy-conservative interface flux and non-positive for an entropy-stable interface flux, and the third line contains volume production terms which are zero for an entropy-conservative volume flux. Under such conditions, the resulting entropy balance becomes
\begin{equation}
\frac{\mathrm{d}}{\mathrm{d} t} I_\physelem[\entropy]
\leq
\sum_{j=0}^N \omega_{j} \left[
(J\entflux^1)^{*}_{0j,0j^+} -(J\entflux^1)^{*}_{Nj,Nj^+}\right]
+ \sum_{i=0}^N \omega_{i} \left[
(J\entflux^2)^{*}_{i0,i0^+} - (J\entflux^2)^{*}_{iN,iN^+}\right],
\end{equation}
which holds as an equality when the volume and interface fluxes are both entropy conservative. A semi-discrete bound on the integrated entropy analogous to \cref{eq:entropy_inequality} can therefore be derived under the same assumptions invoked for global mass conservation in \cref{sec:conservation} through the cancellation of numerical entropy flux contributions at element interfaces.

\subsection{Well balancing}

It is often desired for a numerical method to exactly preserve certain equilibrium states in order to minimize the error in discretizations of problems corresponding to small perturbations of such equilibria. For the shallow water equations, the ``lake at rest'' or ``atmosphere at rest'' state is an important example of such an equilibrium, corresponding to a steady state given everywhere on $S$ by
\begin{equation}\label{eq:equilibrium}
\totalheight = \mathrm{constant}, \quad \vec{v} = \vec{0},
\end{equation}
where $\totalheight \coloneqq h + \topography$ denotes the total surface height, corresponding to a balance between the pressure gradient and the bottom topography term. The following theorem establishes that the proposed scheme preserves such an equilibrium.

\begin{theorem}\label{thm:well_balancing}
Assuming that the bottom topography $\topography$ is continuous, the discretization in \cref{eq:dgsem} with the volume flux and source term introduced in \cref{lem:consistent_flux_covariant} and the surface flux in \cref{eq:flux_ec_llf} is well balanced in the sense that any solution satisfying \cref{eq:equilibrium} results in $\mathrm{d}\state{u}_{ij}/\mathrm{d}t = 0$.
\end{theorem}

\begin{proof}
Proceeding similarly to the proof of \cref{lem:consistent_flux_covariant}, we begin with the strong form in \cref{eq:dgsem_strong} and again note that the interface terms vanish for continuous bottom topographies. Rewriting the volume terms as in \cref{eq:strong_form_diff_1,eq:strong_form_diff_2}, we then set the contravariant velocity components $v^1$ and $v^2$ both to zero and combine the contributions from $h$ and $\topography$ to obtain
\begin{equation}\label{eq:covariant_zero_velocity}
\frac{\mathrm{d}}{\mathrm{d}t}\state{u}_{ij} = - \sum_{m=0}^N\begin{pmatrix}
0\\ g(G^{11}h)_{ij}D_{im}\totalheight_{mj} \\
g(G^{21}h)_{ij}D_{im}\totalheight_{mj}
\end{pmatrix}
- \sum_{m=0}^N\begin{pmatrix}
0\\ g(G^{12}h)_{ij}D_{jm}\totalheight_{im} \\ g(G^{22}h)_{ij}D_{jm}\totalheight_{im}
\end{pmatrix}.
\end{equation}
Finally, we can use \cref{eq:rowsum} to show that both terms on the right-hand side of \cref{eq:covariant_zero_velocity} vanish when the surface height $h + \topography$ is constant.
\end{proof}

\begin{remark}
The assumption of continuous bottom topography is used in the proof of \cref{thm:well_balancing} to show that the interface terms in \cref{eq:dgsem_strong} vanish for constant $h+\topography$.
For discontinuous bottom topographies (\ie $\topography_L \neq \topography_R$ across an interface), these cancellations generally do not hold for the present discretization. In such contexts, well-balanced schemes can be obtained, for example, using hydrostatic reconstruction procedures based on the work of Audusse \etal \cite{audusse_hydrostatic_reconstruction_04} and related well-balancing approaches for discontinuous topography (see, for example, Fjordholm \etal \cite{fjordholm_well_balanced_11}).
\end{remark}

\section{Numerical experiments}\label{sec:numerical_experiments}

The numerical experiments presented in this section are performed using \texttt{TrixiAtmo.jl},\footnote{\url{https://github.com/trixi-framework/TrixiAtmo.jl}} an open-source package written in Julia \cite{bezanson_julia_17} extending the \texttt{Trixi.jl} simulation framework \cite{ranocha_trixi_22} to atmospheric flow problems. The scripts required to produce and post-process the results in this paper are provided within the article's reproducibility repository \cite{reproducibility_repository}. All initial conditions and problem data are prescribed in SI units (kg-m-s), and, following \cite{williamson1992standard}, we make use of the following values for the Earth's radius, rotation rate, and gravitational acceleration:
\begin{equation}
a = 6.37122 \times 10^6 \ \mathrm{m}, \quad \Omega = 7.292 \times 10^{-5} \ \text{s}^{-1}, \quad g = 9.80616\ \mathrm{m} \, \mathrm{s}^{-2}.
\end{equation}
Taking $\cartbasis_z$ to be the Earth's rotation axis, the Coriolis parameter is given by $f= 2\Omega y/a = 2\Omega\sin\theta$ for all cases considered here. The mesh is constructed as described in \cref{sec:mapping}, with the corner vertices placed according to a uniform subdivision of the cube mapped onto the sphere using an equiangular gnomonic projection (see, for example, Ronchi \etal \cite{ronchi_cubed_sphere_96} as well as earlier work by Sadourny \cite{sadourny_cubed_sphere_72}), and all metric terms are computed exactly. Similarly to \cite{bao_flux_form_dg_spherical_swe_14}, we consider the nominal resolution for such cubed-sphere grids to be the average equatorial distance between quadrature points, which can be computed as
\begin{equation}\label{eq:nominal_resolution}
\text{Nominal resolution} \coloneqq \frac{\pi}{2}\frac{a}{MN},
\end{equation}
where $M$ denotes the number of elements per direction on each face of the cube and $N$ denotes the polynomial degree of the tensor-product spectral-element approximation. Time integration is performed using the fourth-order, low-storage explicit Runge--Kutta method of Carpenter and Kennedy \cite{carpenter_kennedy_rk_94} implemented in \texttt{DifferentialEquations.jl} \cite{rackauckas_julia_diffeq_17}. We make use of an adaptive time step based on the Courant--Friedrichs--Lewy (CFL) condition, as given by
\begin{equation}\label{eq:time_step}
\Delta t = C \min_{\physelem\subset S}\min_{i,j=0}^N \left(\frac{2}{N+1} \frac{1} {\lvert\lambda_{ij}^1\rvert +\lvert\lambda_{ij}^2\rvert}\right),
\end{equation}
where $C > 0$ is the Courant number, $\lambda_{ij}^1$ and $\lambda_{ij}^2$ correspond to the nodal values of the contravariant wave speeds in \cref{eq:wave_speeds}, and we note that the quantity in parentheses in \cref{eq:time_step} is an approximation of the ratio of the local mesh spacing to the maximum wave speed in physical space, which we have adapted from Ranocha \etal \cite{ranocha_optimized_rk_22, ranocha_error_based_step_size_dg_25}. In order to ensure that the temporal discretization error is negligible relative to the spatial discretization error and that the robustness of the fully discrete algorithm is not impacted by the CFL condition for the explicit temporal discretization, we take $C = 0.1$ for all numerical experiments in this section. The remainder of this section will consider a sequence of four test cases, which will be used to assess the accuracy and robustness of the proposed schemes as well as to support the theoretical analysis presented in \cref{sec:entropy_stable_covariant}. We will present results for an entropy-conservative formulation (abbreviated in this section as ``EC'') in which both the volume and interface fluxes are computed using \cref{eq:flux_ec_covariant}, as well as an entropy-stable formulation (abbreviated as ``ES'') in which the volume and interface fluxes are given by \cref{eq:flux_ec_covariant} and \cref{eq:flux_ec_llf}, respectively.

\subsection{Unsteady solid-body rotation}

The first case we consider is an unsteady analytical solution to the spherical shallow water equations proposed in \cite[Example~3]{laeuter_unsteady_analytical_swe_05}. To obtain such a solution, we introduce the rotating frame
\begin{equation}
\vec{b}_x(t) \coloneqq \cos(\Omega t)\cartbasis_x + \sin(\Omega t)\cartbasis_y, \quad
\vec{b}_y(t) \coloneqq -\sin(\Omega t)\cartbasis_x + \cos(\Omega t)\cartbasis_y, \quad
\vec{b}_z(t) \coloneqq \cartbasis_z,
\end{equation}
as well as the vector-valued function $\vec{\varphi}(\vec{x},t) \coloneqq (\vec{x} \cdot \vec{b}_x(t)) \cartbasis_x + (\vec{x} \cdot \vec{b}_y(t)) \cartbasis_y + (\vec{x} \cdot \vec{b}_z(t)) \cartbasis_z$. Defining the Earth's rotation vector as $\vec{\Omega} \coloneqq \Omega\cartbasis_z$, prescribing the bottom topography as
\begin{equation}
\topography(\vec{x}) \coloneqq \frac{1}{2g}\big(\vec{\Omega} \cdot \vec{x}\big)^2,
\end{equation}
and defining the fixed vector $\vec{c} \coloneqq -\sin\alpha\cartbasis_x + \cos\alpha\cartbasis_y$, it can be shown that the following surface height and velocity fields solve the spherical shallow water equations:
\begin{equation}\label{eq:exact_solution}
\totalheight(\vec{x},t) = \frac{1}{g}\left(-\frac{(\vec{\Omega}\cdot \vec{x} +
\vref \,\vec{\varphi}(\vec{c},t) \cdot \vec{x}/a)^2}{2} +
\frac{(\vec{\Omega} \cdot \vec{x})^2}{2} + K\right), \quad
\vec{v}(\vec{x},t) = \vref \, \vec{\varphi}(\vec{c},t) \times \frac{\vec{x}}{a}.
\end{equation}
In our numerical experiments, we take values of $\alpha = \pi/4$, $\vref = 2\pi a / (12 \cdot 86400 \ \mathrm{s})$, and $K = 133681 \ \mathrm{m}^2/\mathrm{s}^2$ for consistency with the results in \cite{laeuter_unsteady_analytical_swe_05}. The numerical solution is initialized on each element by evaluating the state vector $\state{u}$ in terms of $\totalheight(\vec{x},0)$, $\vec{v}(\vec{x},0)$, and $\topography(\vec{x})$ at each node position $\vec{x} = \vec{X}(\xi_i,\xi_j)$ in physical space, where \cref{eq:global_local_coordinates} is used to transform the velocity vector from global Cartesian to local contravariant components.
\par Using $\totalheight_{\mathrm{exact}}$ to denote the analytical height field given in \cref{eq:exact_solution}, we compute the $L^2$ error in such a quantity at $t = 5 \text{ days}$ following the conventions suggested in \cite{williamson1992standard}, wherein we employ the quadrature formula in \cref{eq:global_integral} and normalize to obtain
\begin{equation}\label{eq:l2_error}
\text{Normalized $L^2$ height error} \coloneqq \frac{I_S[(\totalheight-\totalheight_{\mathrm{exact}})^2]^{1/2}}{I_S[(\totalheight_{\mathrm{exact}})^2]^{1/2}}.
\end{equation}
Holding the polynomial degree $N$ fixed and successively doubling the number of elements per dimension on each face of the cubed sphere, we observe that the ES scheme converges at an optimal rate of $N+1$, while the EC scheme converges at a rate of $N$ for odd polynomial degrees and $N+1$ for even polynomial degrees, as shown for $N = 3$ and $N = 4$ in \cref{fig:unsteady_solid_body_rotation_convergence_N3} and \cref{fig:unsteady_solid_body_rotation_convergence_N4}, respectively. Holding the number of elements fixed and varying the polynomial degree, we observe exponential convergence in \cref{fig:unsteady_solid_body_rotation_p_refine_M4} for both the EC and ES formulations. In all cases, the error values are smaller for the ES scheme than for the EC scheme. Such results are consistent with the expected convergence behaviour for discontinuous spectral-element discretizations in Euclidean space.

\begin{figure}[!t]
\centering
\begin{subfigure}{0.333\textwidth}
\includegraphics[width=\textwidth]{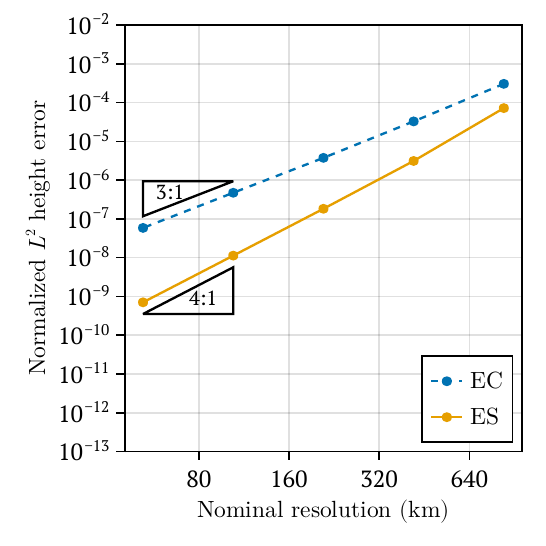}
\subcaption{Convergence with mesh refinement for $N=3$}\label{fig:unsteady_solid_body_rotation_convergence_N3}
\end{subfigure}\hfill
\begin{subfigure}{0.333\textwidth}
\includegraphics[width=\textwidth]{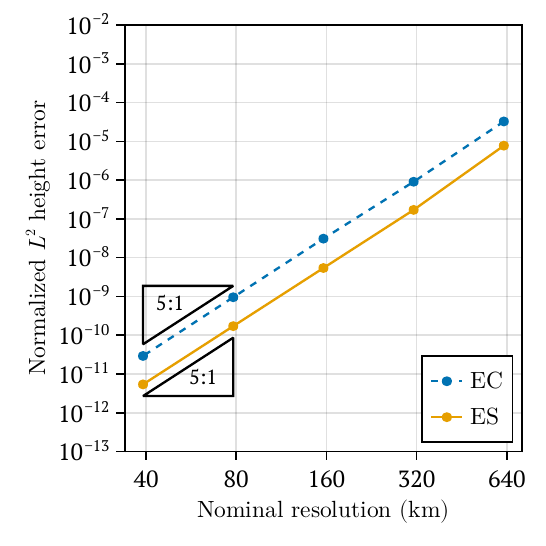}
\subcaption{Convergence with mesh refinement for $N=4$}\label{fig:unsteady_solid_body_rotation_convergence_N4}
\end{subfigure}\hfill
\begin{subfigure}{0.333\textwidth}
\includegraphics[width=\textwidth]{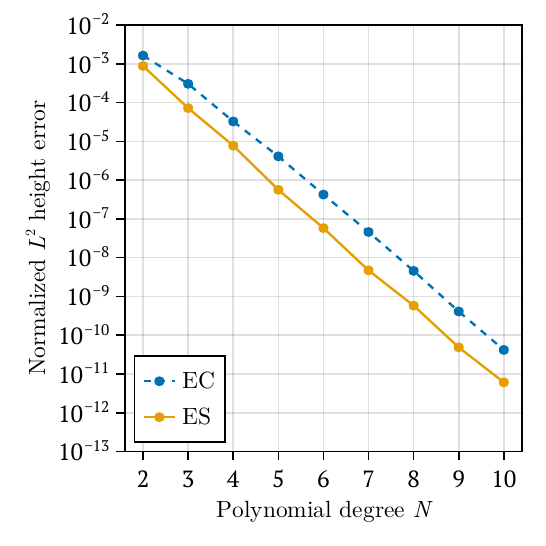}
\subcaption{Convergence with polynomial degree for $M=4$}\label{fig:unsteady_solid_body_rotation_p_refine_M4}
\end{subfigure}%
\caption{Convergence results for the unsteady solid-body rotation problem at $t = 5\text { days}$}\label{fig:unsteady_solid_body_rotation}
\end{figure}

\subsection{Flow over an isolated mountain}\label{sec:isolated_mountain}

Our second test case involves an initially zonal flow over an isolated mountain. Following \cite[Case~5]{williamson1992standard} we consider a mountain of height $\topography_{\mathrm{ref}} = 2000 \ \text{m}$ centred at a longitude of $\lambda_0 = -\pi/2$ and a latitude of $\theta_0 = \pi/6$, as given by
\begin{equation}\label{eq:mountain_orography}
\topography(\lambda,\theta) \coloneqq
\topography_{\mathrm{ref}} \left(1 - \sqrt{\min(R^2, (\lambda-\lambda_0)^2 + (\theta-\theta_0)^2)}/R\right),
\end{equation}
where we take $R = \pi/9$. Representing the velocity field in terms of a spherical basis as $\vec{v} = u\vec{i} + v\vec{j}$, where $u$ and $v$ denote the zonal and meridional velocity components and $\vec{i}$ and $\vec{j}$ denote the longitudinal and latitudinal basis vectors, respectively, we prescribe the initial values of $\totalheight$, $u$, and $v$ as functions of longitude and latitude given by
\begin{equation}\label{eq:isolated_mountain_ic}
\totalheight_0(\lambda,\theta) \coloneqq h_{\mathrm{ref}} - \frac{1}{g}\left(a \Omega \vref + \frac{1}{2} \vref^2\right)\sin^2\theta, \quad
u_0(\lambda,\theta) \coloneqq \vref \cos\theta, \quad
v_0(\lambda,\theta) \coloneqq 0 \ \mathrm{m/s},
\end{equation}
where we take $h_{\mathrm{ref}} = 5960 \ \mathrm{m}$. In our numerical experiments, we consider both the case of $\vref = 0 \ \mathrm{m/s}$, for which a well-balanced discretization is expected to maintain the initial condition up to roundoff error, as well as the case of $\vref = 20 \ \mathrm{m/s}$ described in \cite{williamson1992standard}. For initial velocity fields specified in zonal and meridional components as in \cref{eq:isolated_mountain_ic}, we first transform to a Cartesian representation as
\begin{equation}
\begin{pmatrix}
v_x \\ v_y \\ v_z
\end{pmatrix}
=
\begin{pmatrix}
-\sin\lambda & -\cos\lambda\sin\theta \\
\cos\lambda & - \sin\lambda \sin\theta \\
0 & \cos\theta
\end{pmatrix}
\begin{pmatrix}
u \\ v
\end{pmatrix}
\end{equation}
before using \cref{eq:global_local_coordinates} to obtain the contravariant components. We do not apply any smoothing to the sharp orography in \cref{eq:mountain_orography}; however, we note that the use of the collocation derivative operator to compute the bottom topography gradient within a flux-differencing formulation effectively approximates $\topography$ through polynomial interpolation on each element prior to differentiation.
\par Using $(N, M) = (3, 20)$, as was chosen in several previous studies, including \cite[Section~4.3]{bao_flux_form_dg_spherical_swe_14},
and beginning with the case of $\vref = 0 \ \mathrm{m/s}$, we first note that the initial time derivative of the state vector is on the order of machine precision at all nodes for both the EC and ES variants of the scheme, consistent with \cref{thm:well_balancing}. Plotting the temporal growth of the normalized $L^2$ error over a period of 15 days in \cref{fig:well_balanced_l2_h_evolution_N3M20}, we observe a slow deviation from the initial balanced state due to roundoff error accumulation; such error growth occurs more slowly for the ES variant than for the EC variant due to the dissipative effect of the local Lax--Friedrichs numerical flux. In order to assess the conservation properties of the schemes with respect to mass and total energy (\ie mathematical entropy), we consider the case of $\vref = 20 \ \text{m/s}$ and evaluate the metrics
\begin{equation}\label{eq:normalized_mass_entropy}
\text{Normalized mass change} \coloneqq \frac{I_S[h]- I_S[h_0]}{I_S[h_0]}, \quad
\text{Normalized entropy change}\coloneqq \frac{I_S[\entropy] - I_S[\entropy_0]}{I_S[\entropy_0]},
\end{equation}
where $h_0$ and $\entropy_0$ are the layer depth and mathematical entropy fields evaluated at the initial solution state. Plotting the evolution of such quantities in \cref{fig:isolated_mountain_mass_evolution_N3M20} and \cref{fig:isolated_mountain_entropy_evolution_N3M20}, we see that $I_S[h]$ and $I_S[\entropy]$ remain constant up to roundoff error for the EC scheme, reflecting the fact that the spatial discretization is conservative of mass and entropy as a consequence of \cref{thm:mass_balance,thm:entropy_balance_covariant}, while the small time step renders the effect of the temporal discretization on such properties negligible. For the ES scheme, $I_S[h]$ remains constant up to roundoff error, but $I_S[\entropy]$ decays monotonically, reflecting discrete mass conservation and entropy stability. These results are consistent with the theoretical analysis in \cref{sec:entropy_stable_covariant}.
\par
\begin{figure}[!t]
\centering
\begin{subfigure}{0.333\textwidth}
\includegraphics[width=\textwidth]{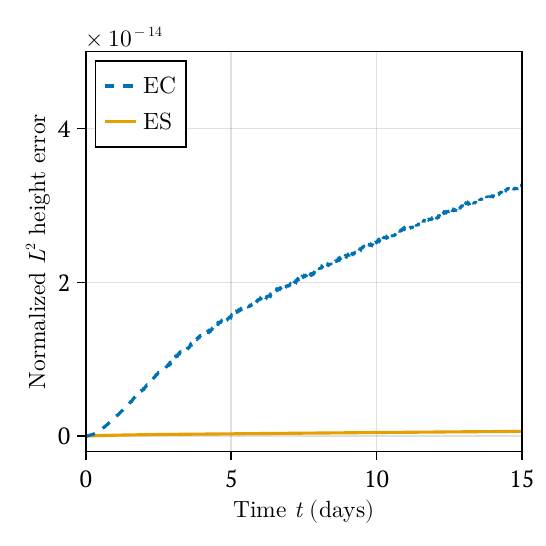}
\subcaption{Normalized $L^2$ height error with $V = 0 \text{ m/s}$}\label{fig:well_balanced_l2_h_evolution_N3M20}
\end{subfigure}\hfill
\begin{subfigure}{0.333\textwidth}
\includegraphics[width=\textwidth]{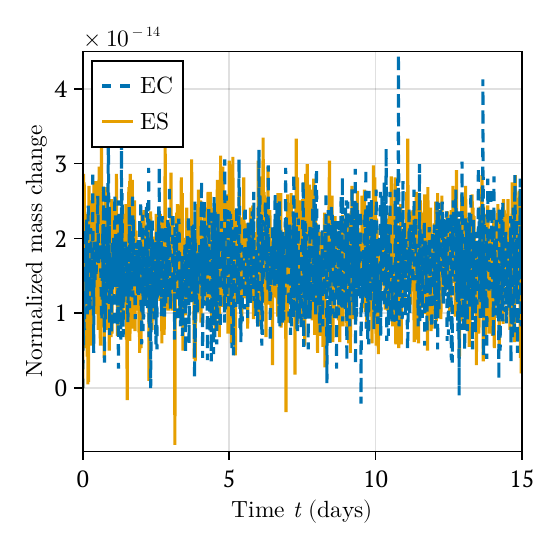}
\subcaption{Normalized mass change with $V = 20 \text{ m/s}$}\label{fig:isolated_mountain_mass_evolution_N3M20}
\end{subfigure}\hfill
\begin{subfigure}{0.333\textwidth}
\includegraphics[width=\textwidth]{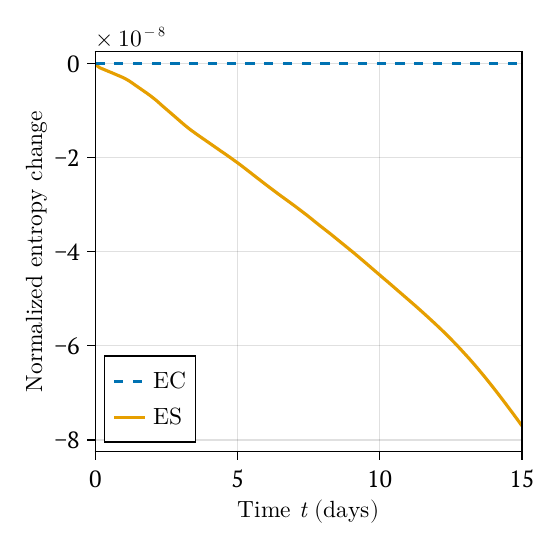}
\subcaption{Normalized entropy change with $V = 20 \text{ m/s}$}\label{fig:isolated_mountain_entropy_evolution_N3M20}
\end{subfigure}%
\caption{Error growth, mass conservation, and total energy (\ie mathematical entropy) conservation or dissipation for the isolated mountain problem with $(N,M) = (3,20)$}\label{fig:isolated_mountain_evolution}
\end{figure}

In \cref{fig:isolated_mountain_contours}, we plot the relative vorticity field, which is given by
\begin{equation}\label{eq:relative_vorticity}
\zeta \coloneqq \varepsilon^{ij}\nabla_iv_j = \frac{1}{J}\big(\partial_1 v_2 - \partial_2 v_1\big),
\end{equation}
at $t = 7 \text{ days}$, considering the case of $\vref = 20 \text{ m/s}$. The resulting contours are qualitatively similar to those presented in \cite[Fig.\ 9]{gaudreault_high_order_shallow_water_cubed_sphere_22} and \cite[Fig.\ 11]{bao_flux_form_dg_spherical_swe_14}. While some oscillatory behaviour is observed in the EC case, likely due to the absence of numerical dissipation, such oscillations are nevertheless smaller than those observed in \cite[Fig.\ 11b]{bao_flux_form_dg_spherical_swe_14} for a standard flux-form DG method at the same polynomial degree and resolution, despite the fact that their method made use of a dissipative interface flux. For our ES formulation using the interface flux in \cref{eq:flux_ec_llf}, which employs the same dissipation term and wave speed estimate used in \cite{bao_flux_form_dg_spherical_swe_14}, such oscillations are not present, suggesting that they may be the result of spurious numerical behaviour which is suppressed through the use of a flux-differencing ES formulation rather than a standard DG scheme. Further comparisons to results obtained using standard DG methods will be presented in \cref{sec:barotropic_instability,sec:rossby}.

\begin{figure}
\centering
\begin{subfigure}{0.5\textwidth}
\centering
\includegraphics[width=\textwidth]{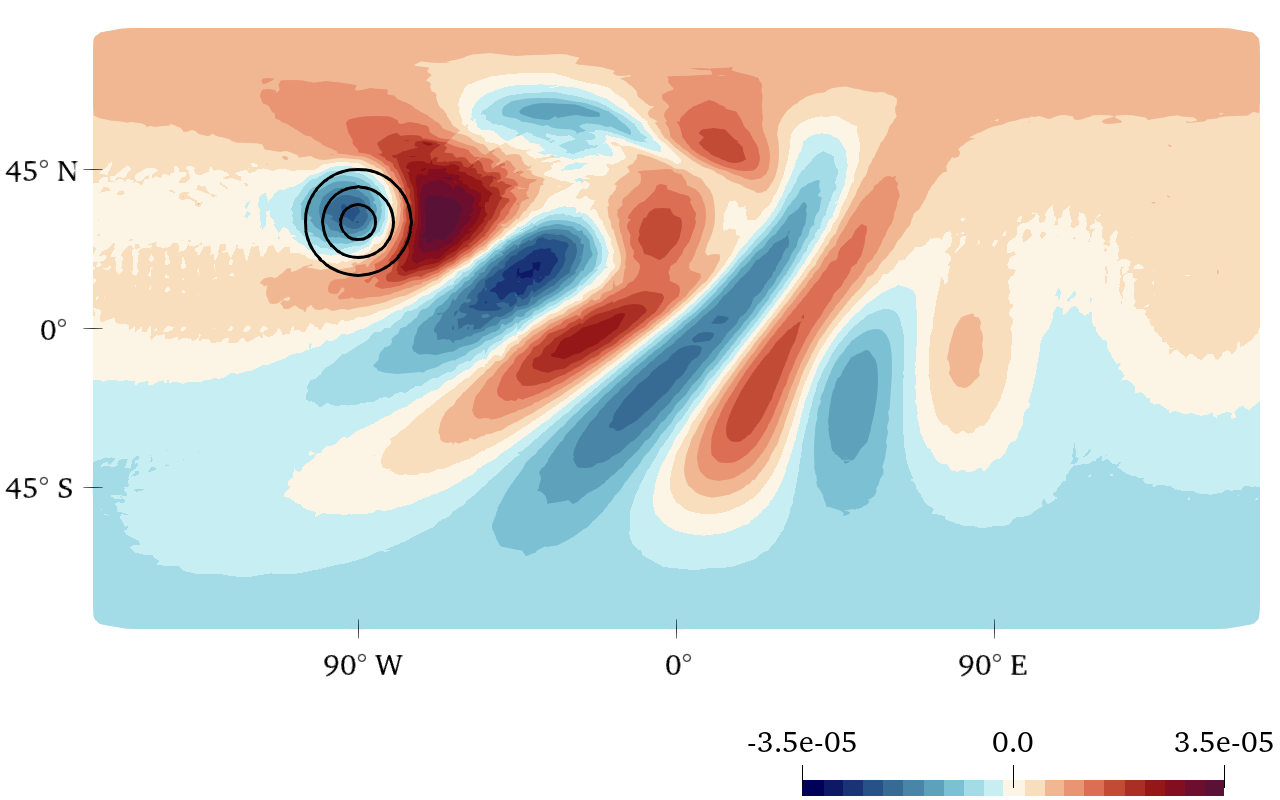}
\caption{EC with $(N, M) = (3,20)$}
\end{subfigure}%
\hfill
\begin{subfigure}{0.5\textwidth}
\includegraphics[width=\textwidth]{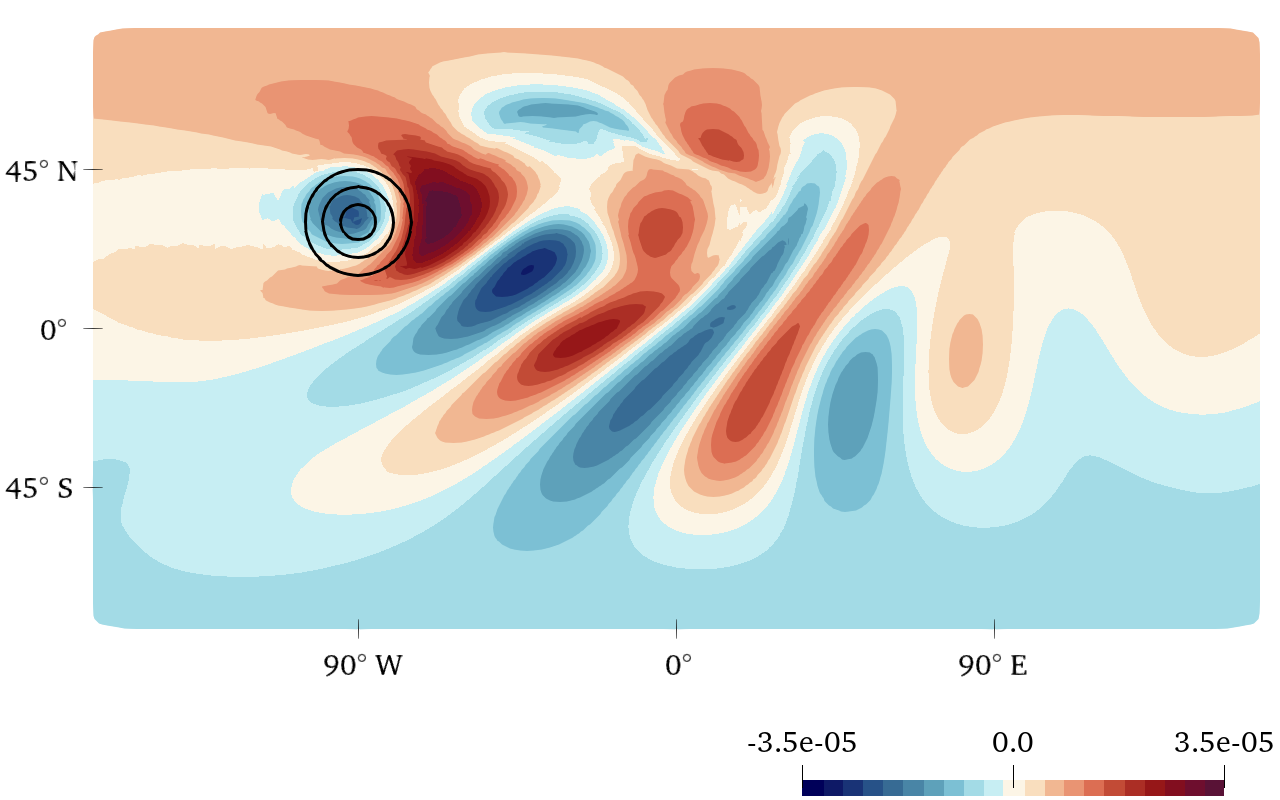}
\centering
\caption{ES with $(N, M) = (3,20)$}
\end{subfigure}%
\caption{Relative vorticity field ($\mathrm{s}^{-1}$) at $t = 7 \text{ days}$ for the isolated mountain problem with $\vref = 20 \text{ m/s}$, where topography contours at $500 \text{ m}$, $1000 \text{ m}$, and $1500 \text{ m}$ are shown in black}\label{fig:isolated_mountain_contours}
\end{figure}

\subsection{Barotropic instability}\label{sec:barotropic_instability}

The third case we consider is a barotropic instability initiated by a perturbation applied to a mid-latitude jet, as proposed in \cite{galewsky_barotropic_instability_04}. The initial velocity field is prescribed as a purely zonal flow given by
\begin{equation}
u_0(\lambda,\theta) \coloneqq \begin{cases}
\dfrac{\vref \exp((\theta - \theta_0)^{-1}
(\theta - \theta_1)^{-1})}{\exp(-4 (\theta_1 - \theta_0)^{-2})}, & \quad \theta_0 < \theta < \theta_1, \\
0, & \quad \text{otherwise},
\end{cases}
\end{equation}
and $v_0(\lambda,\theta) = 0 \ \text{m/s}$, where $\vref = 80 \ \mathrm{m}/\mathrm{s}$, $\theta_0 = \pi/7$, and $\theta_1 = \pi/2 - \theta_0$. Taking $h_{\mathrm{ref}} = 10158 \ \mathrm{m}$ as in \cite{baldauf_dg_shallow_water_covariant_20} so as to achieve a global mean layer depth of approximately $10000 \ \mathrm{m}$ as specified in \cite{galewsky_barotropic_instability_04}, we compute the unperturbed balanced height field as
\begin{equation}
h_{\mathrm{bal}}(\lambda,\theta) \coloneqq h_{\mathrm{ref}} -
\frac{a}{g} \int_{-\pi/2}^{\theta} u_0(\lambda,\theta')\left(2\Omega\sin(\theta') +
u_0(\lambda,\theta')\frac{\tan(\theta')}{ a} \right)\, \mathrm{d}\theta',
\end{equation}
where the integral is evaluated numerically using \texttt{QuadGK.jl} \cite{johnson_quadgk_13}, employing an adaptive scheme based on a composite 15-point Gauss--Kronrod quadrature rule (which itself embeds a seven-point Legendre--Gauss quadrature rule) with a relative tolerance of $\sqrt{\varepsilon}$, where $\varepsilon \approx 2.22 \times 10^{-16}$ denotes machine epsilon in double-precision arithmetic. Adding a small perturbation to the depth field to trigger the instability, we obtain
\begin{equation}\label{eq:perturbation}
h_0(\lambda, \theta) = \begin{cases}
h_{\mathrm{bal}}(\lambda,\theta) + \delta h \cos\theta \exp(-(\lambda/\alpha)^2)
\exp(-((\theta_2 -\theta)/\beta)^2), & \quad -\pi < \lambda < \pi,\\
h_{\mathrm{bal}}(\lambda,\theta), & \quad \text{otherwise},
\end{cases}
\end{equation}
where $\alpha = 1/3$, $\beta = 1/15$, $\theta_2 = \pi/4$, and we consider the unperturbed case of $\delta h = 0 \ \mathrm{m}$ as well as the perturbed case of $\delta h = 120 \ \mathrm{m}$, as recommended in \cite{galewsky_barotropic_instability_04}.
\par
Considering a polynomial degree of $N=3$ and successively refined grids with $M = 16$, $M=32$, and $M=64$, which, according to \cref{eq:nominal_resolution}, correspond to nominal resolutions of approximately 208 km, 104 km, and 52 km, respectively, we evolve the unperturbed solution forward in time and plot the evolution of the normalized $L^2$ error (based on the deviation from the balanced initial state) for the EC and ES schemes over a period of 12 days in \cref{fig:barotropic_instability_evolution}. As the well balancing property established in \cref{thm:well_balancing} applies only to the rest-state equilibrium in \cref{eq:equilibrium}, the discretization is not expected to preserve the initial condition for all time, and we indeed observe significant deviation in all cases due to numerical perturbations which are introduced through discretization error, roundoff error, as well as the representation of the initial condition within a finite-dimensional approximation space. However, the onset of such rapid error growth is shifted later in time as the mesh is refined, and the integrated total energy $I_S[\entropy]$ is nevertheless conserved for the EC scheme and dissipated for the ES scheme, in accordance with \cref{thm:entropy_balance_covariant}.
\par
\begin{figure}[!t]
\centering
\begin{subfigure}{0.333\textwidth}
\centering
\includegraphics[width=\textwidth]{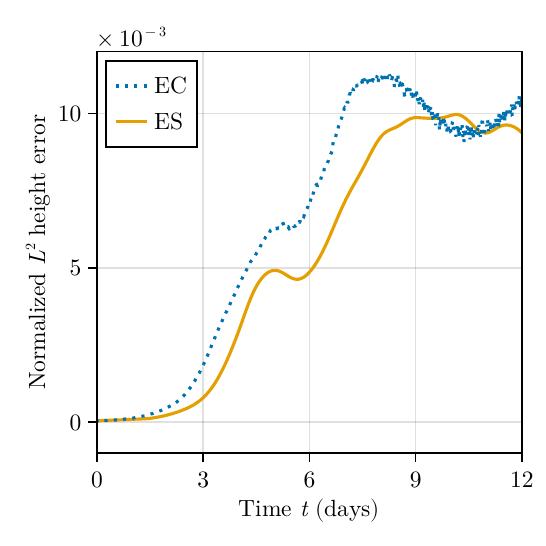}
\subcaption{Normalized $L^2$ height error with $(N, M) = (3, 16)$}
\end{subfigure}%
\hfill
\begin{subfigure}{0.333\textwidth}
\includegraphics[width=\textwidth]{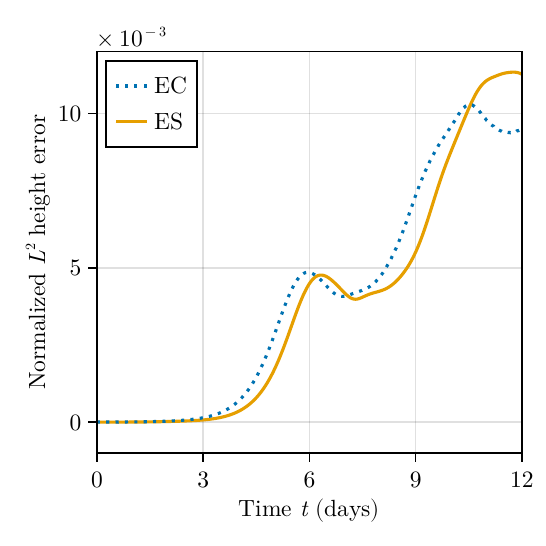}
\subcaption{Normalized $L^2$ height error with $(N, M) = (3, 32)$}
\end{subfigure}%
\hfill
\begin{subfigure}{0.333\textwidth}
\includegraphics[width=\textwidth]{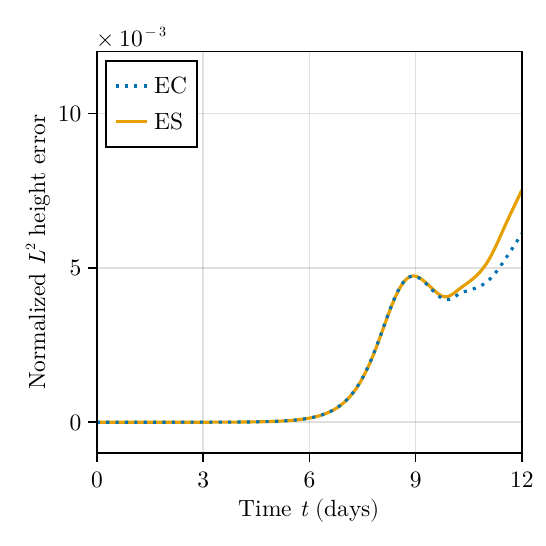}
\subcaption{Normalized $L^2$ height error with $(N, M) = (3, 64)$}
\end{subfigure} \\
\begin{subfigure}{0.333\textwidth}
\centering
\includegraphics[width=\textwidth]{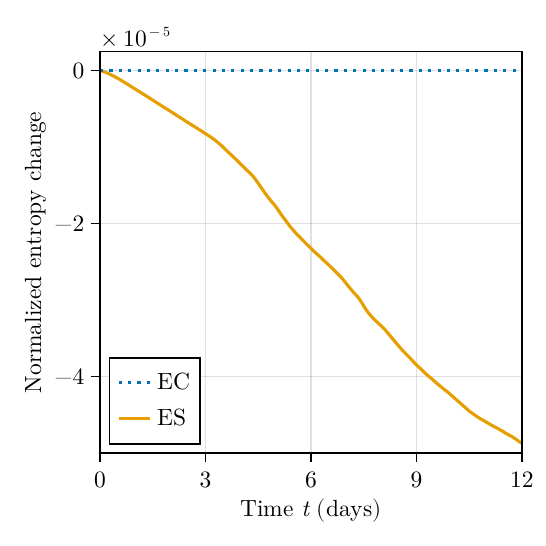}
\subcaption{Normalized entropy change with $(N, M) = (3, 16)$}
\end{subfigure}%
\hfill
\begin{subfigure}{0.333\textwidth}
\includegraphics[width=\textwidth]{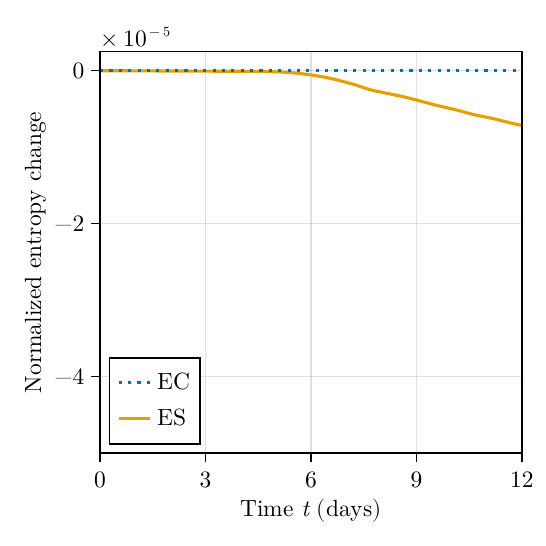}
\subcaption{Normalized entropy change with $(N, M) = (3, 32)$}
\end{subfigure}%
\hfill
\begin{subfigure}{0.333\textwidth}
\includegraphics[width=\textwidth]{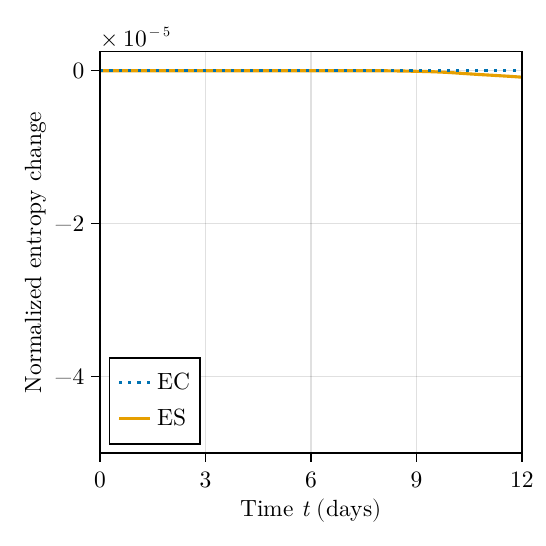}
\subcaption{Normalized entropy change with $(N, M) = (3, 64)$}
\end{subfigure}%
\caption{Error growth and total energy (\ie mathematical entropy) conservation or dissipation for the unperturbed barotropic instability problem}\label{fig:barotropic_instability_evolution}
\end{figure}

While we have thus far used a small time step so as to neglect temporal discretization error in our numerical experiments, it should be emphasized that the proposed discretizations only guarantee entropy conservation or dissipation at the semi-discrete level, requiring additional techniques (\eg relaxation Runge--Kutta methods \cite{ketcheson_relaxation_19,ranocha_relaxation_20}) to ensure fully-discrete entropy stability. To examine the influence of the temporal discretization on mass and entropy conservation, we ran the unperturbed barotropic instability problem using an EC spatial discretization with $(N,M) = (3, 16)$ for fixed time step sizes of $\Delta t \in \{10 \text{ s}, 20 \text{ s}, 40 \text{ s}, 80 \text{ s}, 160 \text{ s}, 320 \text{ s}\}$. Plotting the absolute values of the quantities defined in \cref{eq:normalized_mass_entropy}, we see in \cref{fig:timestep_mass_conservation,fig:timestep_entropy_conservation} that although mass is conserved up to roundoff error for all time step sizes, as expected for linear invariants of ODEs discretized with Runge--Kutta methods \cite{shampine_conservation_laws_odes_86}, the entropy conservation error (which we note to be primarily dissipative here) increases with time more rapidly for larger time step sizes. Plotting the normalized absolute change in both quantities at $t = 12\text{ days}$ in \cref{fig:timestep_convergence}, we observe fifth-order temporal convergence of the entropy error, which is consistent with Chan's observations in \cite[Fig.\ 2c]{chan_discretely_entropy_conservative_dg_18} for an entropy-conservative spatial discretization of the compressible Euler equations using the same fourth-order low-storage explicit Runge--Kutta method employed in this work.

\begin{figure}[t!]
\begin{subfigure}{0.333\textwidth}
\includegraphics[width=\textwidth]{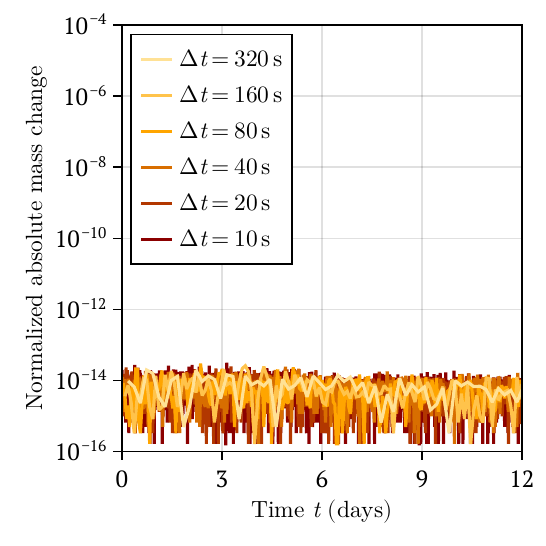}
\subcaption{Mass conservation}\label{fig:timestep_mass_conservation}
\end{subfigure}%
\hfill
\begin{subfigure}{0.333\textwidth}
\includegraphics[width=\textwidth]{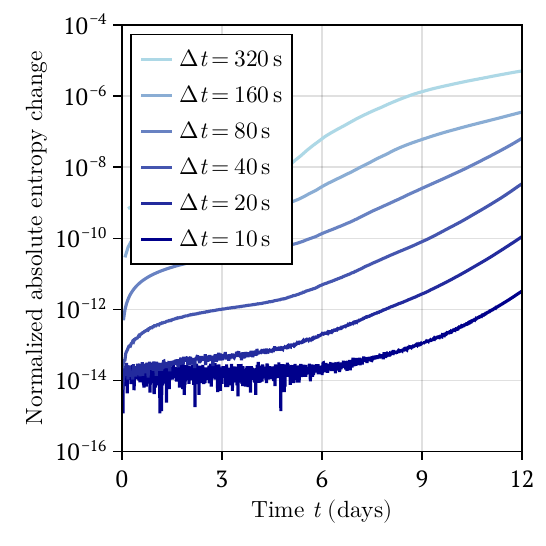}
\subcaption{Entropy conservation}\label{fig:timestep_entropy_conservation}
\end{subfigure}%
\hfill
\begin{subfigure}{0.333\textwidth}
\includegraphics[width=\textwidth]{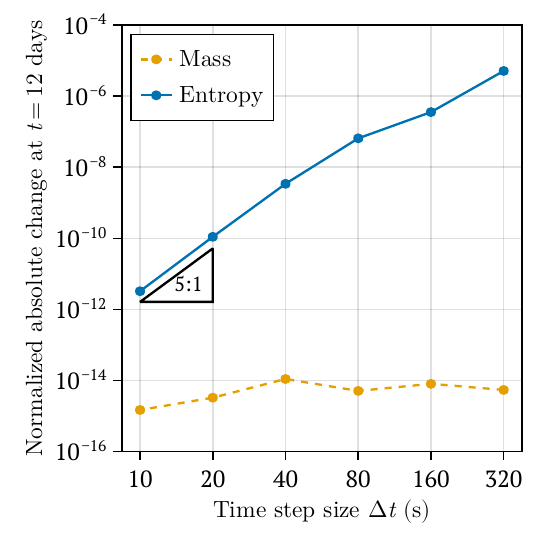}
\subcaption{Convergence with $\Delta t$}\label{fig:timestep_convergence}
\end{subfigure}%
\caption{Effect of time step size on mass and entropy conservation using an EC spatial discretization with $(N,M) = (3,16)$ for the unperturbed barotropic instability problem }\label{fig:barotropic_instability_timestep}
\end{figure}
\par
In \cref{fig:barotropic_instability_contours_unperturbed}, we plot the relative vorticity field given by \cref{eq:relative_vorticity} at $t = 6 \text{ days}$ for an unperturbed initial condition, viewed as an orthographic projection of the northern hemisphere with the $\lambda = 0$ meridian pointed downwards. For the lower-resolution simulations with $M = 16$ and $M=32$, nonphysical flow features (\ie grid imprinting) associated with numerical perturbation growth can be seen, with significantly more high-frequency noise observed for the EC scheme than for the ES scheme. At $M = 64$, the initial condition is preserved quite well, suggesting that such a resolution would likely be sufficient for qualitatively preserving the physical flow features arising from a perturbed initial condition. Comparing the contours in \cref{fig:barotropic_instability_contours_perturbed} resulting from the perturbed initial condition to reference solutions such as those in \cite{galewsky_barotropic_instability_04}, we see that this is indeed the case, with $M=64$ being sufficient to obtain vorticity fields which exhibit the qualitative features observed in the literature, without interference by artificial flow features initiated by numerical perturbations absent in the analytical definition of the initial condition.
\par
Although nominal resolution as defined in \cref{eq:nominal_resolution} is not a perfectly generalizable measure of a scheme's capability to capture fine-scale flow features, particularly when comparing methods with different orders of accuracy, we note that the 52 km nominal resolution obtained for $(N,M) = (3, 64)$ is approximately the same as that used for a standard DG method in \cite[Section~4.2]{bao_flux_form_dg_spherical_swe_14}, where their choice of $(N,M) = (7, 30)$ corresponds to a nominal resolution of approximately 48 km. Furthermore, the elimination of visible grid imprinting between resolutions of 104 km and 52 km is roughly consistent with the findings of St-Cyr \etal \cite[Fig.\ 13]{stcyr_jablonowski_comparason_adaptive_shallow_water_08}, wherein grid imprinting is visible for a continuous spectral-element method at a reported resolution of 1.25$^\circ$ (approximately 139 km) but not at $0.625^\circ$ (approximately 69 km).
\par
\begin{figure}[!t]
\centering
\begin{subfigure}{0.333\textwidth}
\includegraphics[width=\textwidth]{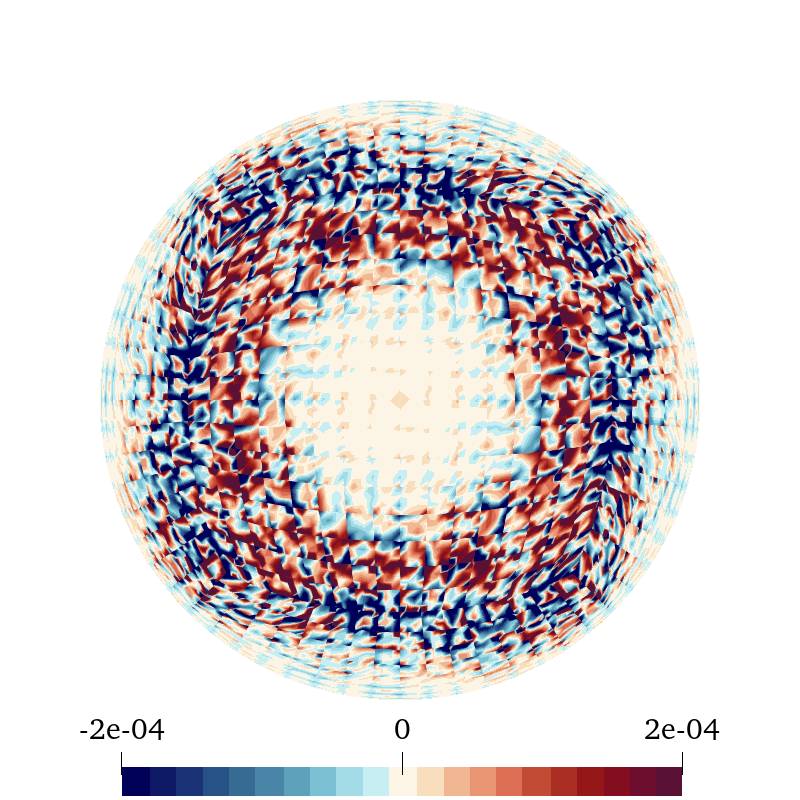}
\subcaption{EC with $(N, M) = (3, 16)$}
\end{subfigure}%
\hfill
\begin{subfigure}{0.333\textwidth}
\includegraphics[width=\textwidth]{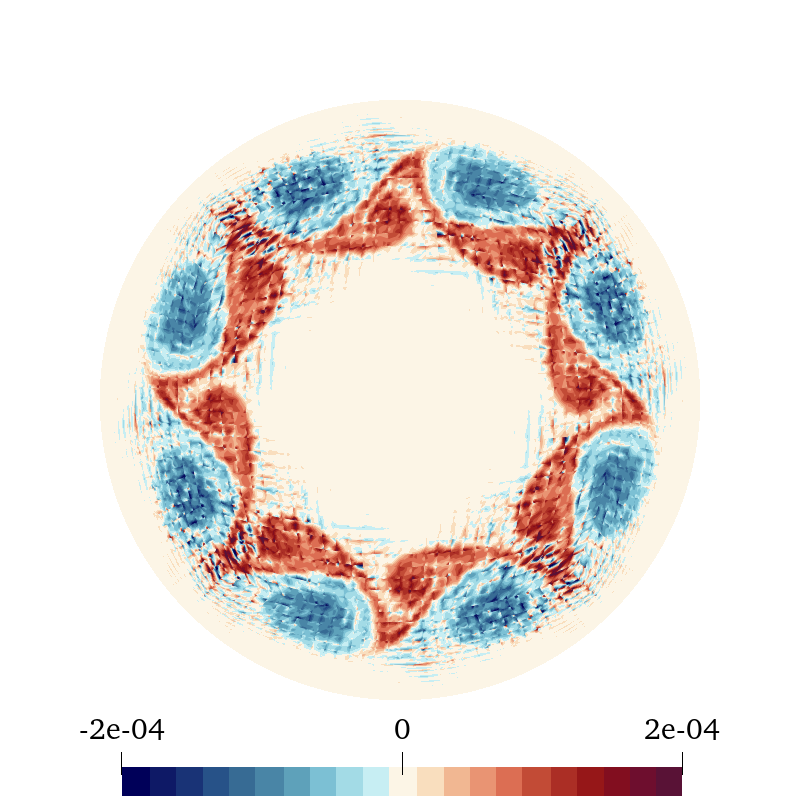}
\subcaption{EC with $(N, M) = (3, 32)$}
\end{subfigure}%
\hfill
\begin{subfigure}{0.333\textwidth}
\includegraphics[width=\textwidth]{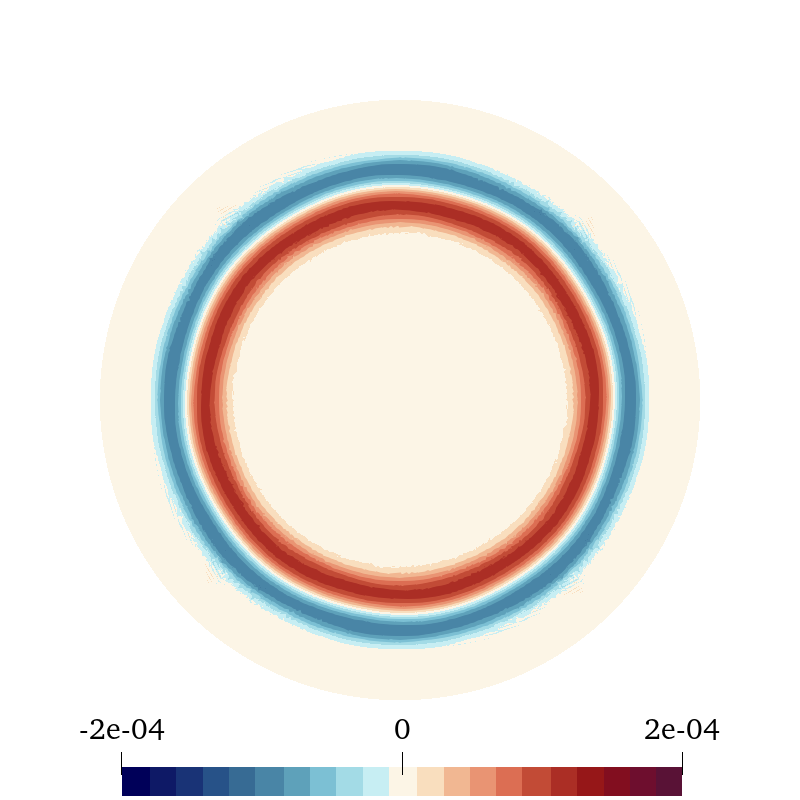}
\subcaption{EC with $(N, M) = (3, 64)$}
\end{subfigure}\\
\begin{subfigure}{0.333\textwidth}
\includegraphics[width=\textwidth]{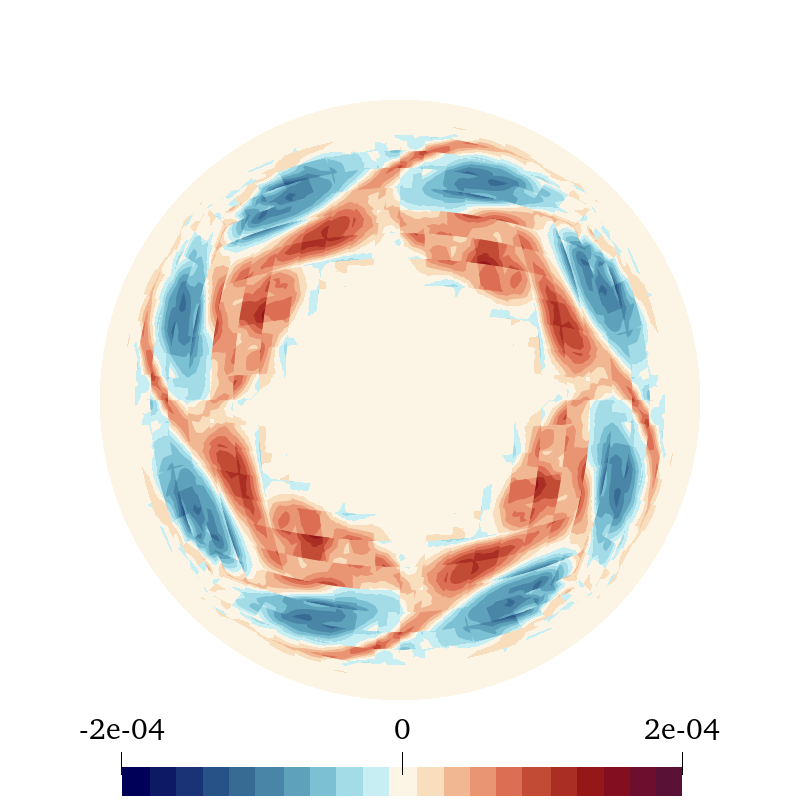}
\subcaption{ES with $(N, M) = (3, 16)$}
\end{subfigure}%
\hfill
\begin{subfigure}{0.333\textwidth}
\includegraphics[width=\textwidth]{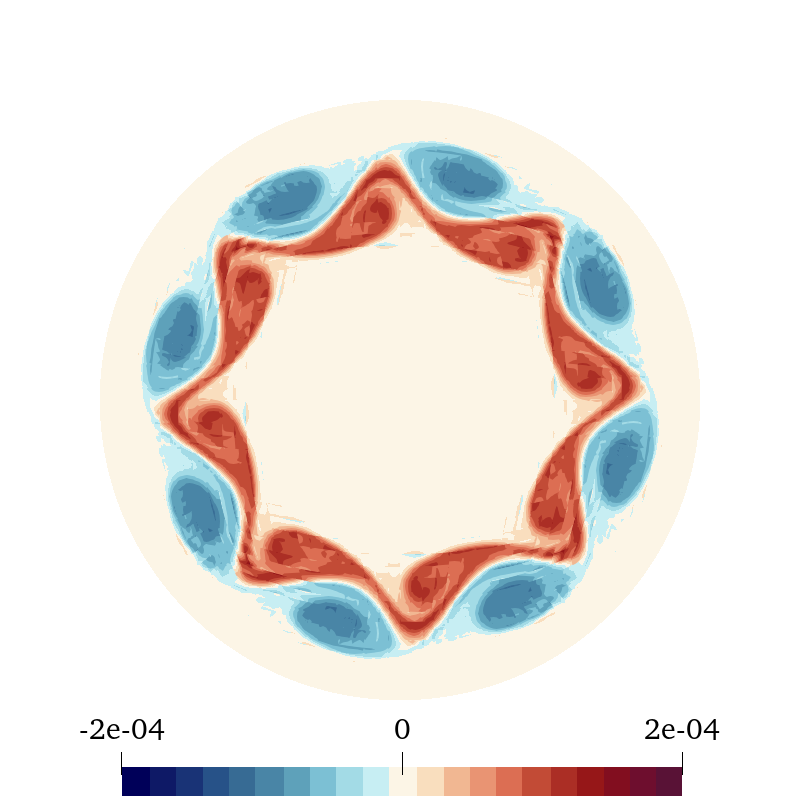}
\subcaption{ES with $(N, M) = (3, 32) $}
\end{subfigure}%
\hfill
\begin{subfigure}{0.333\textwidth}
\includegraphics[width=\textwidth]{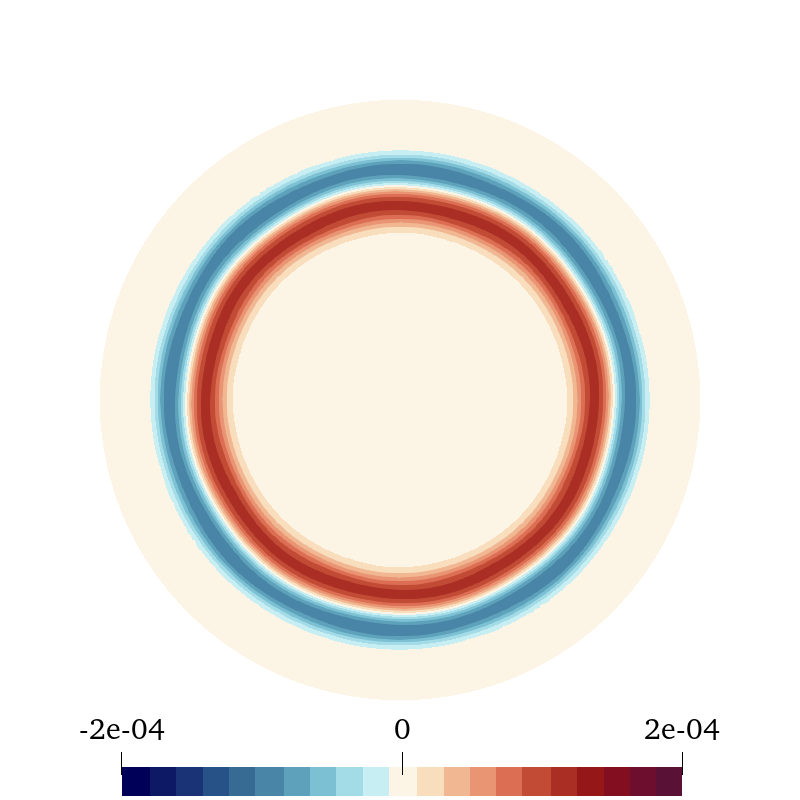}
\subcaption{ES with $(N, M) = (3, 64)$}
\end{subfigure}%
\caption{Relative vorticity field ($\text{s}^{-1}$) at $t = 6 \text{ days}$ for the unperturbed barotropic instability problem }\label{fig:barotropic_instability_contours_unperturbed}
\end{figure}

\begin{figure}[!t]
\centering
\begin{subfigure}{0.333\textwidth}
\includegraphics[width=\textwidth]{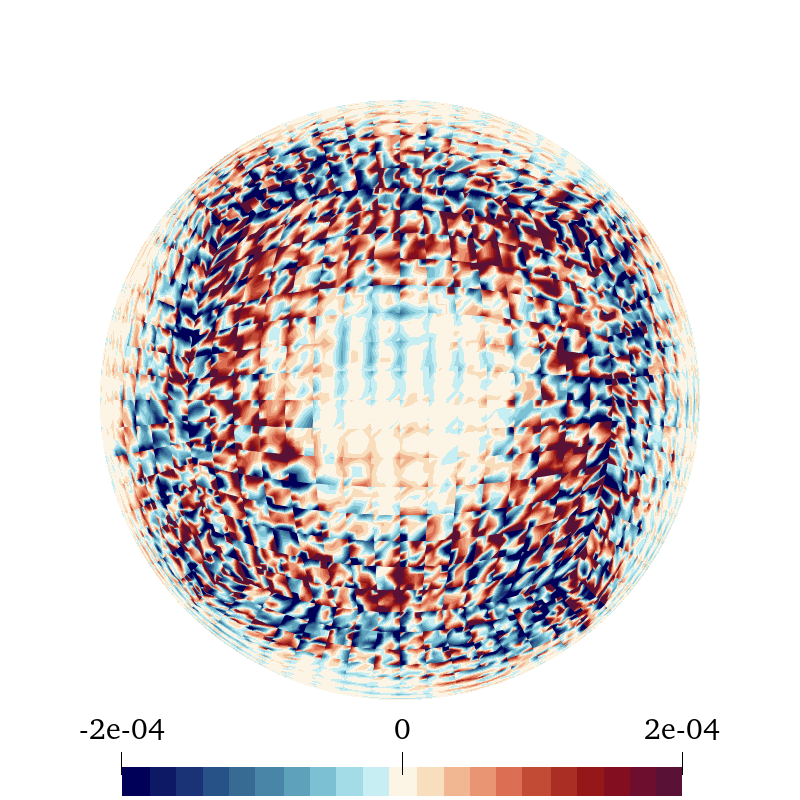}
\subcaption{EC with $(N, M) = (3, 16)$}
\end{subfigure}%
\hfill
\begin{subfigure}{0.333\textwidth}
\includegraphics[width=\textwidth]{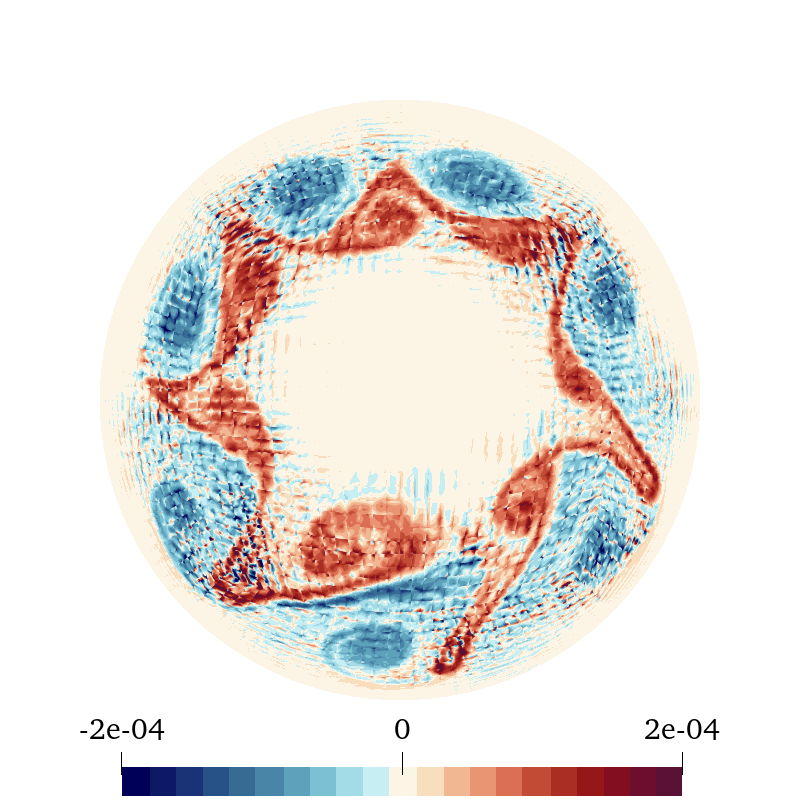}
\subcaption{EC with $(N, M) = (3, 32)$}
\end{subfigure}%
\hfill
\begin{subfigure}{0.333\textwidth}
\includegraphics[width=\textwidth]{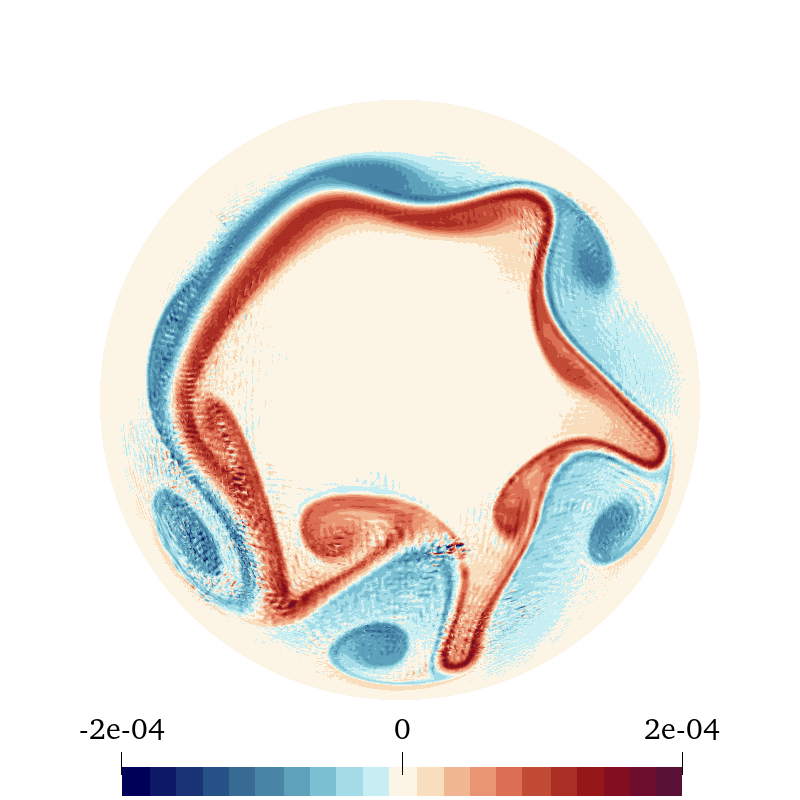}
\subcaption{EC with $(N, M) = (3, 64)$}
\end{subfigure}\\
\begin{subfigure}{0.333\textwidth}
\includegraphics[width=\textwidth]{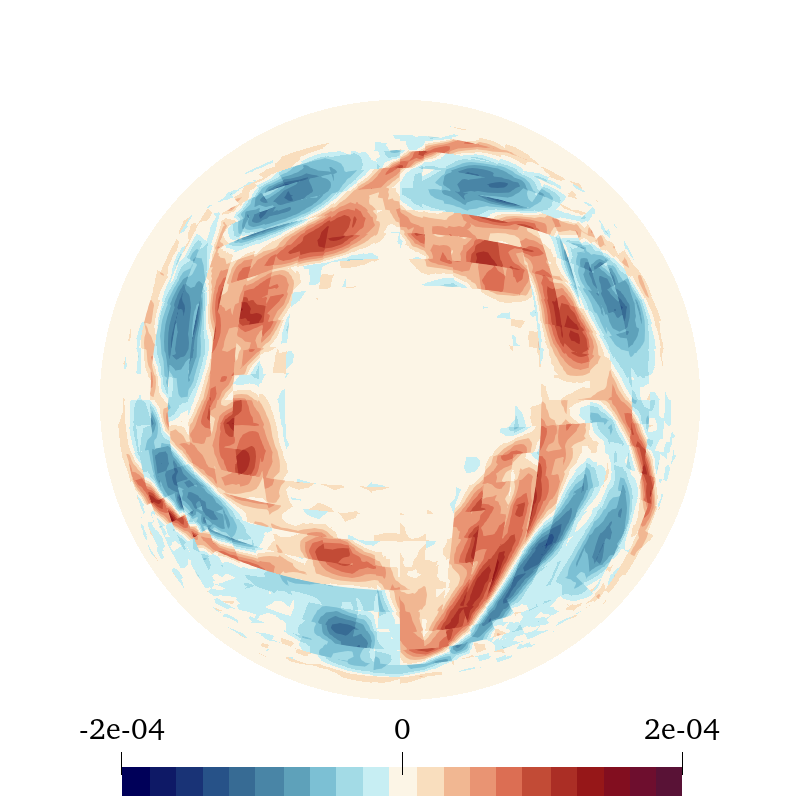}
\subcaption{ES with $(N, M) = (3, 16)$}
\end{subfigure}%
\hfill
\begin{subfigure}{0.333\textwidth}
\includegraphics[width=\textwidth]{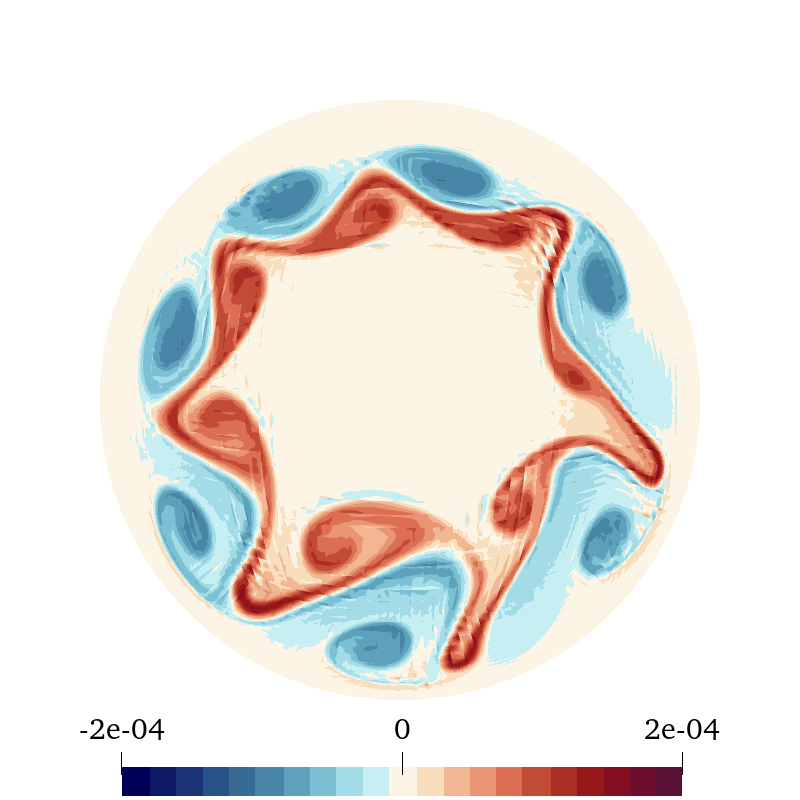}
\subcaption{ES with $(N, M) = (3, 32)$}
\end{subfigure}%
\hfill
\begin{subfigure}{0.333\textwidth}
\includegraphics[width=\textwidth]{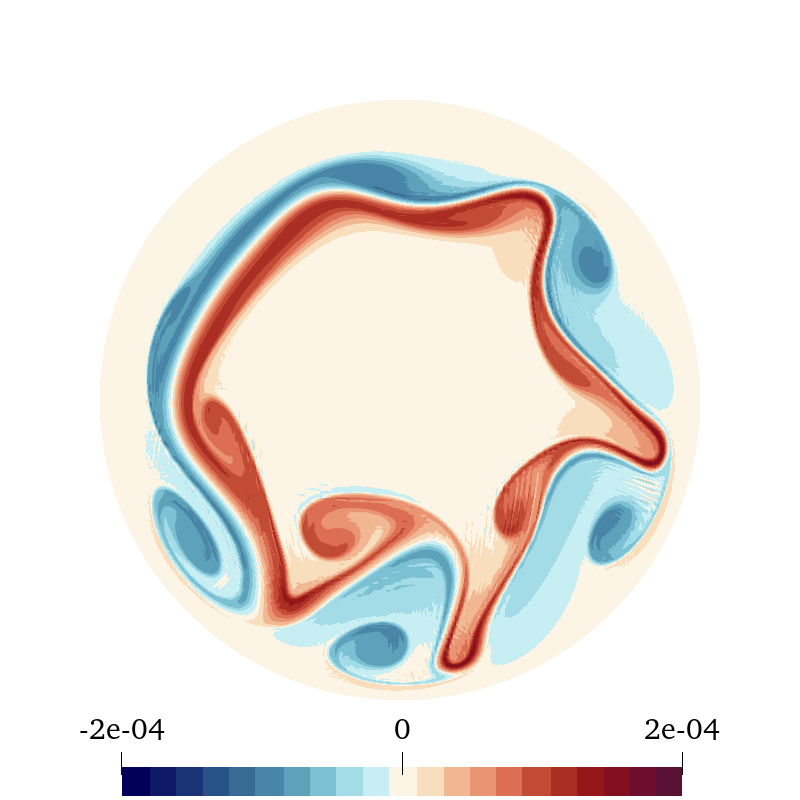}
\subcaption{ES with $(N, M) = (3, 64)$}
\end{subfigure}%
\caption{Relative vorticity field ($\text{s}^{-1}$) at $t = 6 \text{ days}$ for the perturbed barotropic instability problem}\label{fig:barotropic_instability_contours_perturbed}
\end{figure}

To examine the influence of the discretization on the solution quality more closely, we show in \cref{fig:barotropic_instability_closeup} close-up views of a portion of the vorticity field for the EC and ES schemes, as well as for a standard DG method using the same cubed-sphere grid, tensor-product polynomial approximation space, and collocated LGL quadrature.
The standard DG method discretizes \cref{eq:pde_reference_space} using the weak formulation in \cref{eq:dg_standard}, without skew-symmetric splitting or flux differencing, and employs a local Lax--Friedrichs interface flux given by
\begin{equation}\label{eq:flux_standard_llf}
(J\state{f}^j)^{*}_{L,R} \coloneqq \frac{1}{2} \left[(J\state{f}^j)_L + (J\state{f}^j)_R \right] - \frac{J_L}{2}\max(\lambda_L^j, \lambda_R^j)(\state{u}_R - \state{u}_L),
\end{equation}
where we note that the dissipation term is identical to that employed for the ES scheme in \cref{eq:flux_ec_llf}. While the EC scheme exhibits oscillatory behaviour of a similar nature to that observed in the isolated mountain case due to the lack of any mechanism to dissipate high-frequency noise, it is nevertheless remarkable that a dissipation-free numerical method is able to remain stable for a relatively complex model problem for a geophysical flow exhibiting multiscale behaviour and steep solution gradients. These oscillations are greatly reduced through the introduction of the dissipation term to the EC scheme to obtain the ES formulation. While the standard DG method was able to run to completion for this case, we nevertheless observe some spurious behaviour in \cref{fig:barotropic_instability_closeup_standard} which is not present in \cref{fig:barotropic_instability_closeup_es}, qualitatively illustrating the benefit of the flux-differencing formulation in this particular case. We will present quantitative robustness comparisons in the following subsection.

\begin{figure}[t!]
\begin{subfigure}{0.333\textwidth}
\includegraphics[width=\textwidth]{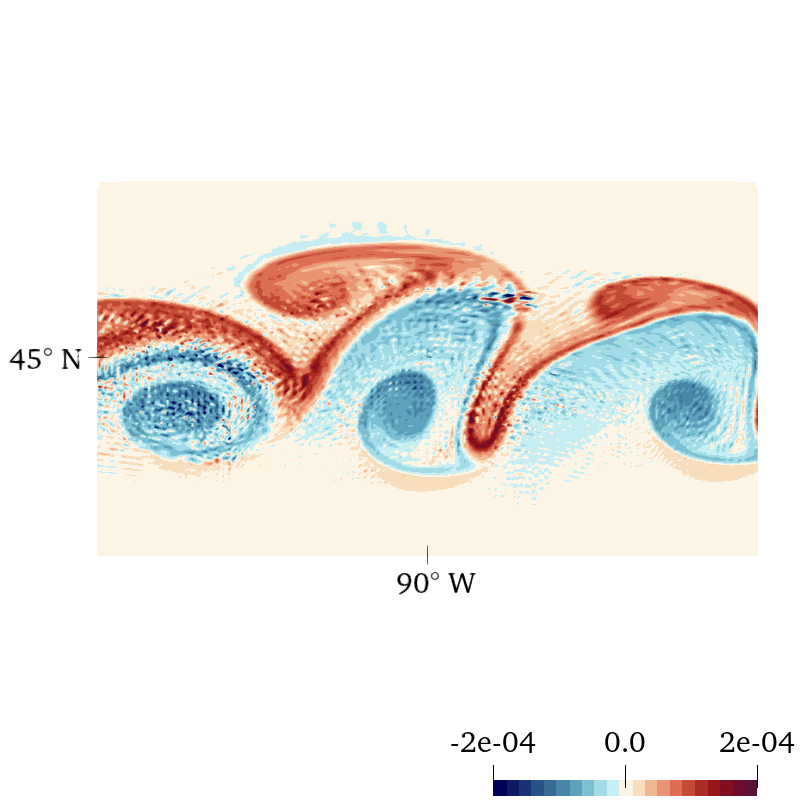}
\subcaption{Entropy-conservative method}\label{fig:barotropic_instability_closeup_ec}
\end{subfigure}%
\hfill
\begin{subfigure}{0.333\textwidth}
\includegraphics[width=\textwidth]{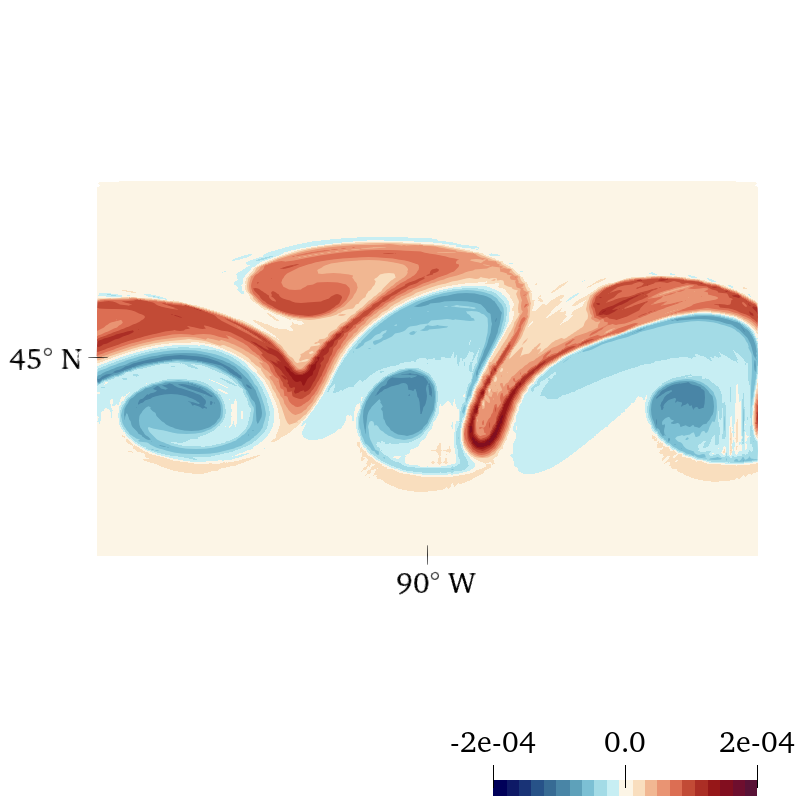}
\subcaption{Entropy-stable method}\label{fig:barotropic_instability_closeup_es}
\end{subfigure}%
\hfill
\begin{subfigure}{0.333\textwidth}
\includegraphics[width=\textwidth]{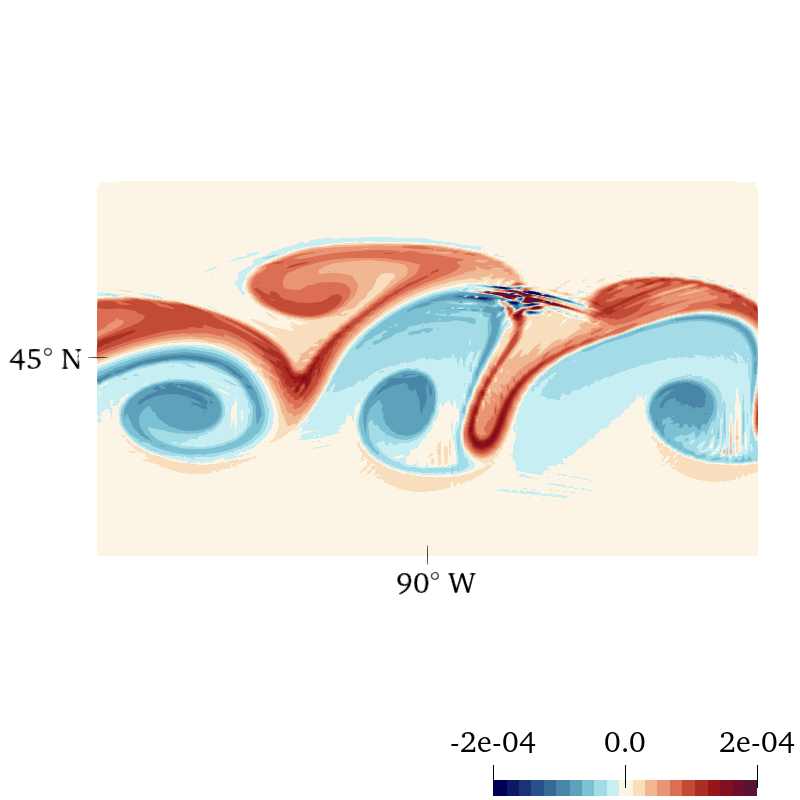}
\subcaption{Standard DG method}\label{fig:barotropic_instability_closeup_standard}
\end{subfigure}%
\caption{Comparison of relative vorticity fields ($\text{s}^{-1}$) at $t = 6 \text{ days}$ with $(N,M) = (3,64)$} \label{fig:barotropic_instability_closeup}
\end{figure}

\subsection{Rossby--Haurwitz wave}\label{sec:rossby}

As a fourth and final test case, we consider a Rossby--Haurwitz wave with wavenumber four, as described in \cite[Case~6]{williamson1992standard}. Taking parameter values of $\omega = K = 7.848 \times 10^{-6} \ \mathrm{s}^{-1}$, $R = 4$, and $h_{\mathrm{ref}} = 8000 \ \mathrm{m}$, the functions
\begin{subequations}
\begin{align}
A(\theta) &\coloneqq \frac{\omega}{2}(2 \Omega+\omega) \cos^2 \theta +
\frac{K^2}{4} \cos^{2 R} \theta\left((R+1) \cos^2\theta + (2 R^2-R-2) -
2\frac{R^2}{\cos^2\theta} \right), \\
B(\theta) &\coloneqq \frac{2(\Omega+\omega) K}{(R+1)(R+2)} \cos ^R \theta\left((R^2+2 R+2) -
(R+1)^2 \cos^2 \theta\right), \\
C(\theta) &\coloneqq \frac{1}{4} K^2 \cos^{2 R} \theta\left((R+1) \cos^2 \theta-(R+2)\right),
\end{align}
\end{subequations}
are used to prescribe the initial depth and spherical velocity components as
\begin{subequations}
\begin{align}
h_0(\lambda,\theta) &= h_{\mathrm{ref}} +
\frac{a^2}{g}\left(A(\theta) + B(\theta)\cos(R\lambda) + C(\theta)\cos(2R\lambda) \right),\\
u_0(\lambda,\theta) & = a \omega \cos \theta+a K \cos ^{R-1} \theta
(R \sin ^2 \theta-\cos ^2 \theta) \cos (R \lambda),\\
v_0(\lambda,\theta) & = -a K R \cos ^{R-1} \theta \sin \theta \sin (R \lambda).
\end{align}
\end{subequations}
While it is common to run such a test case until $t = 7 \ \text{days}$ or $t = 14 \ \text{days}$, we run for 28 days as a more severe stress test of the schemes' robustness for predictions over longer time horizons. Furthermore, as a baseline for such robustness comparisons, we also report results using standard collocated DG methods employing LGL quadrature and a local Lax--Friedrichs numerical flux, as described in \cref{sec:barotropic_instability}.
\par

\begin{figure}[!t]
\begin{subfigure}{0.5\textwidth}
\includegraphics[width=\textwidth]{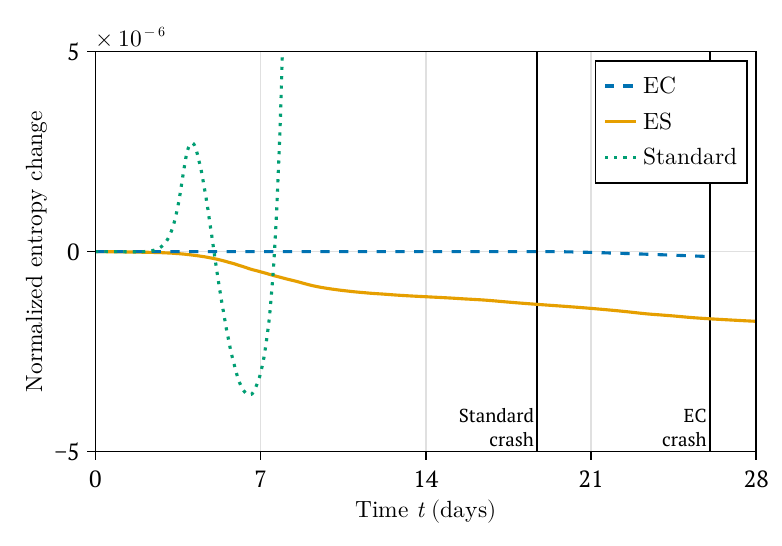}
\subcaption{Normalized entropy change and crash times with $(N, M) = (3, 16)$}\label{fig:rossby_entropy_N3M16}
\end{subfigure}\hfill
\begin{subfigure}{0.5\textwidth}
\includegraphics[width=\textwidth]{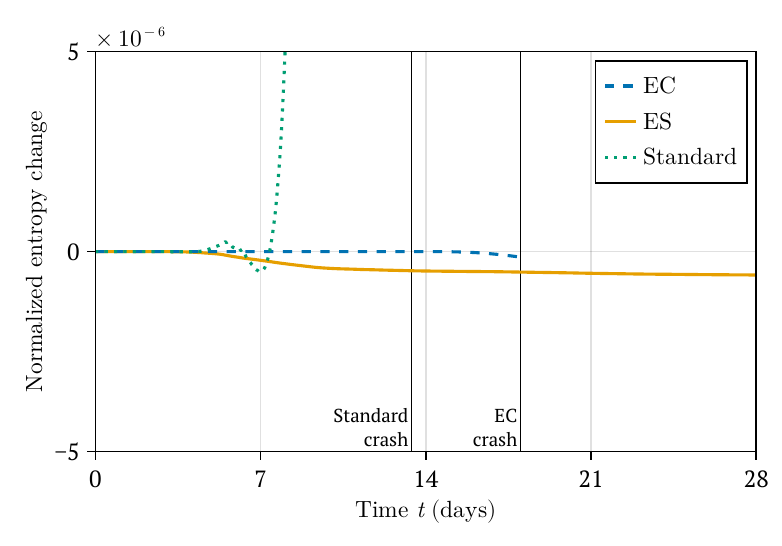}
\subcaption{Normalized entropy change and crash times with $(N, M) = (6, 8)$}\label{fig:rossby_entropy_N6M8}
\end{subfigure}\\
\begin{subfigure}{0.5\textwidth}
\includegraphics[width=\textwidth]{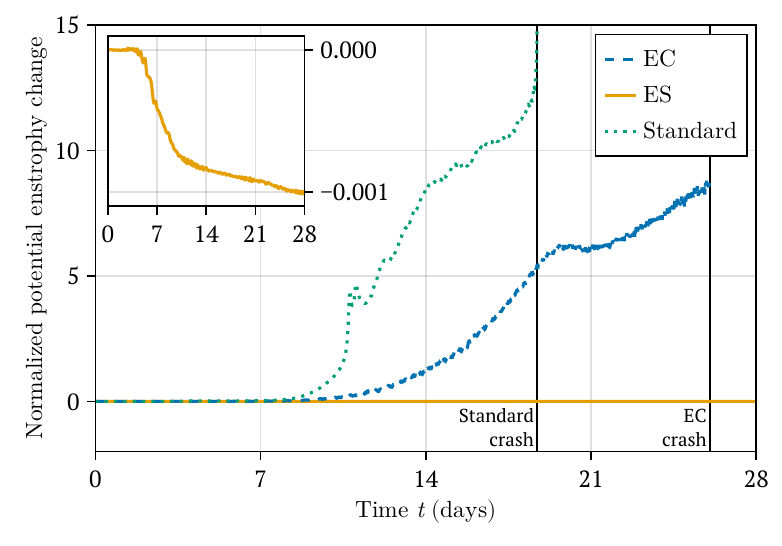}
\subcaption{Normalized potential enstrophy change and crash times with $(N, M) = (3, 16)$}\label{fig:rossby_potential_enstrophy_N3M16}
\end{subfigure}\hfill
\begin{subfigure}{0.5\textwidth}
\includegraphics[width=\textwidth]{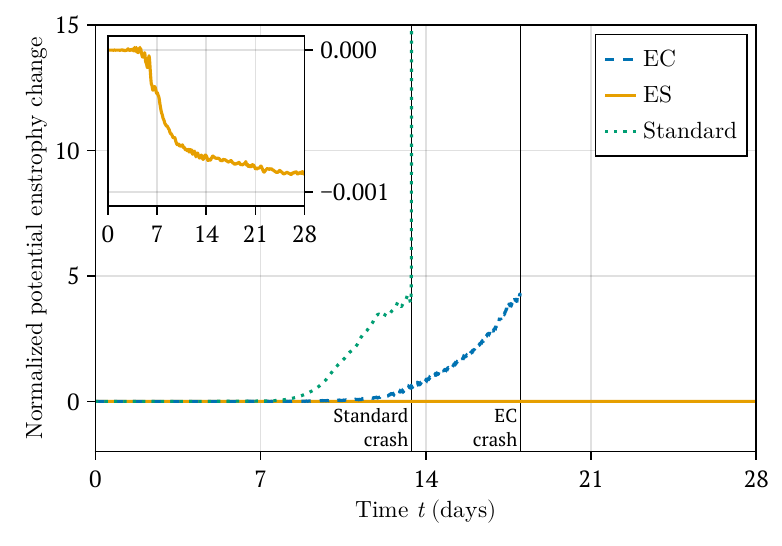}
\subcaption{Normalized potential enstrophy change and crash times with $(N, M) = (6, 8)$}\label{fig:rossby_potential_enstrophy_N6M8}
\end{subfigure}%
\caption{Robustness of EC, ES, and standard collocated DG schemes for the Rossby--Haurwitz wave}
\label{fig:rossby_haurwitz}
\end{figure}

Considering the cases of $(N, M) = (3, 16)$ as well as that of $(N, M) = (6, 8)$, which differ in polynomial degree but correspond to the same nominal resolution of approximately 208 km, we plot the entropy conservation error for the EC, ES, and standard DG schemes in \cref{fig:rossby_entropy_N3M16,fig:rossby_entropy_N6M8}, with vertical lines indicating crash times. When applied to the spherical shallow water equations, the standard DG scheme is conservative of mass, but does not satisfy any discrete entropy conservation or stability properties. We accordingly observe uncontrolled growth in the total energy for the standard DG schemes, with crashes occurring at $t = 18.71 \ \text{days}$ and $t = 13.39 \ \text{days}$ for the $(N, M) = (3, 16)$ and $(N, M) = (6, 8)$ simulations respectively. For both polynomial degrees considered, the EC method is able to run for longer than the corresponding standard DG scheme, but nevertheless eventually crashes due to its inability to damp oscillations, which eventually leads to negative values for the layer depth. As with the standard DG method, the $N=3$ EC scheme runs for longer than the $N=6$ EC scheme at an equal nominal resolution, with crashes occurring at $t = 26.03 \ \text{days}$ and $t = 18.01 \ \text{days}$. Unlike the standard and EC formulations, the ES schemes were capable of running the full 28-day simulations to completion for both $(N, M) = (3, 16)$ and $(N, M) = (6, 8)$, indicating that the local Lax--Friedrichs interface dissipation term in \cref{eq:flux_ec_llf} provided sufficient damping of the oscillations that led to negative depth values when the EC scheme was applied to this problem.
Examining the evolution of the discrete integral of the potential enstrophy $Z \coloneqq (\zeta + f)^2/h$, where $\zeta$ is the relative vorticity defined in \cref{eq:relative_vorticity}, we see in \cref{fig:rossby_potential_enstrophy_N3M16,fig:rossby_potential_enstrophy_N6M8} that the aforementioned oscillatory behaviour manifests for the standard DG and EC schemes in the form of a gradual buildup of grid-scale vorticity, beginning well before the simulations crashed.
While the ES scheme does not guarantee strict conservation or dissipation of potential enstrophy or other vorticity-derived invariants, the integrated potential enstrophy deviated by no more than approximately 0.1\% over the course of each of the 28-day simulations, with such deviations being almost entirely dissipative in nature. These results indicate that the interface dissipation term was sufficient in this case to suppress spurious potential enstrophy production.

\section{Conclusions}

A new discontinuous spectral-element formulation was devised for the rotating shallow water equations with variable bottom topography expressed in covariant form on general unstructured quadrilateral tessellations of curved two-dimensional manifolds. The approach relies on the use of tensor-product summation-by-parts operators within a flux-differencing formulation in reference coordinates derived from a skew-symmetric splitting of the divergence of the momentum flux tensor in covariant form. The resulting methods are shown to be conservative of mass, and, depending on the choice of numerical interface flux, conservative or dissipative of total energy, which serves as a mathematical entropy function for the shallow water system.
\par The proposed schemes were applied to a series of standard model problems for idealized atmospheric dynamics based on the shallow water equations on the sphere, including an unsteady solid-body rotation, a zonal flow over an isolated mountain, a barotropic instability, and a Rossby--Haurwitz wave. Using an entropy-conservative numerical flux at element interfaces, the resulting schemes conserve mass as well as total energy, serving as a robust yet dissipation-free baseline to which numerical dissipation can be added judiciously, for example, to reduce oscillations. When such dissipation is introduced through an entropy-stable interface flux, the resulting schemes guarantee conservation of mass and dissipation of total energy, attain optimal convergence (\ie order $N+1$ for a tensor-product approximation of polynomial degree $N$), and exhibit excellent robustness even under challenging conditions including those involving under-resolved vortical structures and integration over long time horizons.
\par

The advances in this paper represent a first step in the process of developing a new atmospheric dynamical core for numerical weather prediction and climate modelling.
Future work will involve the extension of the proposed approach to the three-dimensional nonhydrostatic compressible Euler equations using an appropriate vertical discretization and a horizontally-explicit vertically-implicit temporal discretization \cite{baldauf_dg_hevi_terrain_following_21}. Additionally, we are investigating extensions to triangular and prismatic elements through the use of multidimensional SBP formulations (as introduced by Hicken \etal \cite{hicken_mdsbp_16} and reviewed within an entropy-stable setting by Chen and Shu \cite{chen_shu_dgsbp_review_20}) as well as entropy-stable tensor-product formulations in collapsed coordinates, which were introduced by Montoya and Zingg \cite{montoya_entropy_stable_24} and extended to prismatic elements by Keim \etal \cite{keim_entropy_stable_hybrid_25}. We would also like to investigate techniques for preserving vorticity-based invariants such as potential enstrophy within the present framework. Additionally, we hope to extend the proposed methodology beyond the context of geophysical fluid dynamics, for example, to covariant formulations of hyperbolic partial differential equations arising in general relativity and relativistic hydrodynamics, as well as to the treatment of curved elements in computational fluid dynamics for engineering applications.

\subsection*{Acknowledgements}

The authors would like to thank Dr.\ Michael Baldauf for the very insightful conversations about the spherical shallow water equations. Tristan Montoya, Gregor J. Gassner, and Andrés M. Rueda-Ramírez acknowledge funding through the German Federal Ministry for Education and Research (BMBF) project ``ICON-DG'' (01LK2315B) of the ``WarmWorld Smarter'' program. Gregor J. Gassner and Andrés M. Rueda-Ramírez acknowledge funding through the Klaus-Tschira Stiftung via the project ``HiFi-Lab'' (00.014.2021). Andrés M. Rueda-Ramírez gratefully acknowledges funding from the Spanish Ministry of Science, Innovation, and Universities through the ``Beatriz Galindo'' grant (BG23-00062)
and from the European Research Council through the Synergies Grant Agreement No.\ 101167322-TRANSDIFFUSE.

\subsection*{Declaration of generative AI in scientific writing}

During proofreading of this article, ChatGPT (OpenAI) was used to assist in identifying spelling, grammar, and clarity issues. All manuscript text was written and thoroughly reviewed by the authors, who take full responsibility for its content.

\end{document}